\documentclass[a4paper,12pt,leqno]{article}
\usepackage{amssymb,amsmath}
\usepackage{xy}
\xyoption{all}
\usepackage{graphicx}
\usepackage{hyperref}

\numberwithin{equation}{section}

\newtheorem{defn}[equation]{Definition}

\newtheorem{corollary}[equation]{Corollary}
\newtheorem{rem}[equation]{Remark}
\newtheorem{exm}[equation]{Example}
\newtheorem{lemma}[equation]{Lemma}
\newtheorem{theorem}[equation]{Theorem}
\newtheorem{notat}[equation]{Notation}
\newtheorem{newpar}[equation]{}

\newtheorem{xdefn}{Definition.}
\newtheorem{xproposition}{Proposition.}
\newtheorem{xcorollary}{Corollary.}
\newtheorem{xrem}{Remark.}
\newtheorem{xexm}{Example.}
\newtheorem{xlemma}{Lemma.}
\newtheorem{xtheorem}{Theorem.}
\newtheorem{xnotat}{Notation.}
\newtheorem{xnewpar}{\it}
\newtheorem{xproof}{{\it Proof. }}
\newtheorem{xproofof}{{\it Proof}}

\newenvironment{definition}{\begin{defn}\em}{\end{defn}}
\newenvironment{remark}{\begin{rem}\em}{\end{rem}}
\newenvironment{example}{\begin{exm}\em}{\end{exm}}

\newenvironment{proof}{\begin{xproof}\em}{\end{xproof}}

\newenvironment{newparagraph*}[1]{\begin{xnewpar}\hspace*{-1.5mm}{#1}. \rm}{\end{xnewpar}}

\newenvironment{definition*}{\begin{xdefn}\em}{\end{xdefn}}
\newenvironment{remark*}{\begin{xrem}\em}{\end{xrem}}
\newenvironment{example*}{\begin{xexm}\em}{\end{xexm}}
\newenvironment{notation*}{\begin{xnotat}\em}{\end{xnotat}}
\newenvironment{proposition*}{\begin{xproposition}}{\end{xproposition}}
\newenvironment{corollary*}{\begin{xcorollary}}{\end{xcorollary}}
\newenvironment{lemma*}{\begin{xlemma}}{\end{xlemma}}
\newenvironment{theorem*}{\begin{xtheorem}}{\end{xtheorem}}

 \def\qed{\hspace{0.3cm}{\rule{1ex}{2ex}}} \newcommand\V{\bigvee} 
 \newcommand\ie{i.e.} \newcommand\eg{e.g.}   
 \newcommand\st{\mid}

\newcommand\cf{\textrm{cf.}}

\newcommand\ds[1]{\widetilde{#1}}
\newcommand\sections{\mathit{\Gamma}}
\newcommand\opens{\operatorname{\mathcal{O}}}

\newcommand\topology{\operatorname{\Omega}}
\newcommand\groupoid{\operatorname{\mathcal{G}}}
\newcommand\spp{\varsigma}
\newcommand\downsegment{{\downarrow}}

\newcommand\ptimes[1]{\mathop{\times}\limits_{#1}}
\newcommand\rs{\operatorname{R}}
\newcommand\ls{\operatorname{L}}
\newcommand\ident{\mathrm{id}}
\newcommand\ipi{\mathcal I}

\newcommand\lcc{\operatorname{{\mathcal L}^{\vee}}}

\newcommand\ideal{\mathfrak I}
\newcommand\idinc{\iota}
\newcommand\imu{\mathfrak m} 
\newcommand\cover[1]{\widehat{#1}}

\begin{document}

\title{Quantales of open groupoids}
\author{\sc M.\ Clarence Protin and Pedro Resende\thanks{Research supported in part by the Funda\c{c}\~ao para a Ci\^encia e a Tecnologia through the Program POCI 2010/FEDER, namely via the doctoral grant SFRH/BD/17823/2004 (first author) and Centro de An\'{a}lise Matem\'{a}tica, Geometria e Sistemas Din\^{a}micos (second author).}}

\date{~}

\maketitle

\begin{abstract}
It is well known that inverse semigroups are closely related to \'etale groupoids. In particular, it has recently been shown that there is a (non-functorial) equivalence between localic \'etale groupoids, on one hand, and complete and infinitely distributive inverse semigroups (abstract complete pseudogroups), on the other. This correspondence is mediated by a class of quantales, known as inverse quantal frames, that are obtained from the inverse semigroups by a simple join completion that yields an equivalence of categories.
Hence, we can regard abstract complete pseudogroups as being essentially ``the same'' as inverse quantal frames, and in this paper we exploit this fact in order to find a suitable replacement for inverse semigroups in the context of open groupoids that are not necessarily \'etale. The interest of such a generalization lies in the importance and ubiquity of open groupoids in areas such as operator algebras, differential geometry and topos theory, and we achieve it by means of a class of quantales, called open quantal frames, which generalize inverse quantal frames and whose properties we study in detail.
The resulting correspondence between quantales and open groupoids is not a straightforward generalization of the previous results concerning \'etale groupoids, and it depends heavily on the existence of inverse semigroups of local bisections of the quantales involved.
\vspace{0.2cm}
\vspace{0.2cm}
\end{abstract}

\maketitle

\tableofcontents

\section{Introduction}\label{introduction}

It is well known that groupoids and inverse semigroups are generalizations of groups which, in particular, cater for more general notions of symmetry \cite{Lawson,Weinstein}. Furthermore, the two concepts are closely related in more than one way, a recurring theme being that from certain topological groupoids one obtains inverse semigroups of ``local bisections'', whereas from suitable inverse semigroups one constructs groupoids of ``germs''. This correspondence is well known and widely used, for instance, in differential topology \cite{MoerdijkMrcun} and operator algebra theory \cite{Renault2}. It is not an equivalence, but it restricts to a (non-functorial) equivalence between topological \'etale groupoids over a space $X$ (that is, \'etale groupoids $G$ whose unit space $G_0$ equals $X$) and complete and infinitely distributive inverse semigroups $S$ acting on $X$ in a way that determines an isomorphism between the lattice of open sets of $X$ and the lattice of idempotents of $S$ \cite{MatsnevResende}.

If $X$ is sober (that is, each irreducible closed set is the closure of a unique singleton subset) the action of $S$ on $X$ is uniquely determined by the chosen isomorphism and, even more generally, one may replace sober spaces by locales \cite{Johnstone,pointless} in order to obtain a bijection between localic \'etale groupoids (\ie, internal groupoids in the category of locales) and complete and infinitely distributive inverse semigroups \cite{aim}.

One is often free to choose whether to work with groupoids or with inverse semigroups (see, \eg, Fell bundles on inverse semigroups rather than groupoids in \cite{Exel}), but such freedom of choice always entails that the underlying groupoid must be \'etale; that is, its domain map (and thus also the codomain map) is a local homeomorphism.
However, there are many situations where non-\'etale groupoids arise naturally. Lie groupoids, for instance, such as the holonomy groupoids of foliations, are in general non-\'etale; but the fact that the domain map is required to be a submersion makes them \emph{open groupoids} in the sense that the domain map $d$ is necessarily open. Something similar can be said of locally compact groupoids in the sense of \cite{Paterson}, for which openness is a topological consequence of the existence of Haar measures. In topos theory, too, the fundamental theorem of Joyal and Tierney \cite{JT} states that any Grothendieck topos is equivalent to the category of equivariant sheaves on an open localic groupoid. Furthermore, this is an important example of how groupoids can be regarded as generalized spaces, or, at least, as presentations of generalized spaces, which is also a common motto in the stacks literature (see, \eg, \cite{stacks}) and throughout noncommutative geometry in the sense of Connes \cite{Connes}.

The importance of open groupoids across mathematics leads one to asking the question of whether a useful algebraic counterpart can be found for them in a way that generalizes the role played by inverse semigroups in relation to \'etale groupoids. A way of addressing this, which we shall pursue in the present paper, is based on the observation that from an inverse semigroup $S$ a quantale $\lcc(S)$ is obtained if we complete $S$ by adding the suprema of all the subsets of $S$, with respect to the natural order of $S$. The quantales obtained in this manner are the \emph{inverse quantal frames} \cite{aim} and they form a category which is equivalent to that of complete and infinitely distributive inverse semigroups. Hence, for many practical purposes, in the context of \'etale groupoids it is irrelevant whether one chooses to work with an inverse semigroup $S$ or instead with its quantale completion $\lcc(S)$.

There is also a direct relation between \'etale groupoids and inverse quantal frames which does not require the mediation of inverse semigroups: if $G$ is a localic \'etale groupoid with multiplication map $m:G_2\to G_1$ ($G_2$ is the pullback $G_1\times_{G_0} G_1$ of the domain and range maps), the sup-lattice $G_1$ itself is canonically equipped with a multiplication, given by the following composition in the category of sup-lattices (see \cite{JT}), where $m_!$ is the direct image homomorphism of $m$ (which exists because $m$ is necessarily open, in fact a local homeomorphism):
\[\xymatrix{G_1\otimes G_1\ar@{->>}[r] &G_2\ar[r]^{m_!}& G_1}\;.\]
The resulting quantale is denoted by $\opens(G)$. It is an inverse quantal frame, and it is isomorphic to the quantale completion $\lcc(\sections(G))$ of the inverse semigroup $\sections(G)$ of local bisections of $G$. For topological \'etale groupoids something analogous holds, with the quantale being simply the topology of $G$ with product given by pointwise multiplication of open sets.

Conversely, a localic groupoid $\groupoid(Q)$ can be directly obtained from an inverse quantal frame $Q$ without any reference to germs or inverse semigroups. In order to see this, let $e$ be the multiplicative unit of $Q$; the down segment $Q_0=\downsegment e$ is a locale, $Q$ is a $Q_0$-$Q_0$-bimodule over it, and the quantale multiplication \[Q\otimes Q\to Q\] factors through the quotient $Q\otimes_{Q_0} Q$ via a sup-lattice homomorphism $\mu$:
\[\xymatrix{Q\otimes Q\ar@{->>}[r]&Q\otimes_{Q_0} Q\ar[r]^-\mu& Q}\;.\]
Then we define a localic groupoid $G=\groupoid(Q)$, with $G_1=Q$ and $G_0=Q_0$ (and $G_2=Q\otimes_{Q_0} Q$), whose multiplication $m$ is defined by $m_!=\mu$. This requires the right adjoint $\mu_*$ to preserve joins, which is not a trivial condition but holds for inverse quantal frames. In \cite{aim} this property of inverse quantal frames is referred to as \emph{multiplicativity}. In addition, a topological groupoid can be obtained from any localic groupoid via the spectrum functor of locales, since this functor has a left adjoint and thus it preserves limits. In particular, if $S$ is a complete and infinitely distributive inverse semigroup, the topological groupoid obtained as the spectrum of the localic groupoid $\groupoid(\lcc(S))$ is exactly the groupoid of germs of $S$ in the usual sense.

We can summarize the above facts by stating that the following diagram is commutative up to isomorphisms of the objects of the categories involved, and moreover $\lcc$ defines an equivalence of categories:

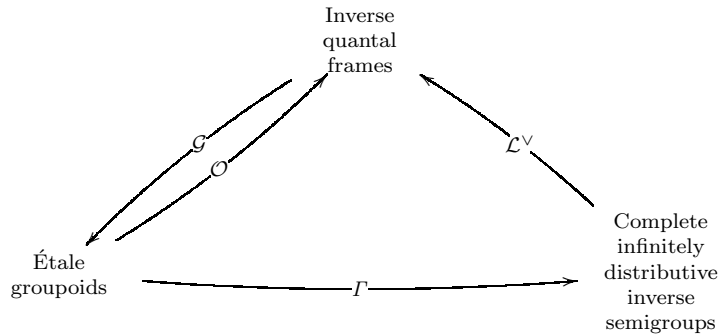
\begin{equation*}
\vcenter{\xymatrix{
&&\begin{minipage}{1.7cm}\begin{center}\scriptsize Inverse quantal frames\end{center}\end{minipage}\ar@/_/[ddll]|{\groupoid}
\\
&&&&~\\
\begin{minipage}{2cm}\begin{center}\scriptsize \'{E}tale groupoids\end{center}\end{minipage}\ar@/_/[rrrr]|{\sections}\ar@/_/[rruu]|{\opens}&&&&
\begin{minipage}{2cm}\begin{center}\scriptsize Complete infinitely distributive inverse semigroups\end{center}\end{minipage}\ar@/_/[uull]|{\lcc}
}} \label{trianglediagram}
\end{equation*}

If $G$ is no longer \'etale but merely an open groupoid, both $\sections(G)$ and $\opens(G)$ can still be defined as before; that is, $\sections(G)$ is the set of continuous local bisections of $G$ and $\opens(G)$ is $G_1$ equipped with the direct image of the multiplication map. Of course, $\opens(G)$ is no longer the join-completion of $\sections(G)$, which certainly does not contain enough information to recover the original groupoid. However, as we shall see, $G$ is still determined up to isomorphism by the quantale $\opens(G)$. The argument is similar to that of \'etale groupoids, but there is a big difference as regards the algebraic characterization of the quantales of the form $\opens(G)$, which now is more complicated because, contrary to inverse quantal frames, the multiplicativity condition is no longer a consequence of a handful of more ``elementary'' axioms. The main aim of this paper is precisely to address this question, and in doing so we shall be led into studying properties, both weaker and stronger than multiplicativity, which are interesting in their own right. We remark that our results also provide a new example of how quantales can be models of generalized notions of space (in this case open groupoids), in the spirit of the earlier works that relate quantales to C*-algebras \cite{KR,Mulveyenc,PR}.

We shall begin, in section \ref{sec:groupoidquantales}, by studying thoroughly a set of simple axioms for (not necessarily unital) quantales that we shall refer to as \emph{open quantal frames}. As we shall see, the unital open quantal frames are precisely the same as the inverse quantal frames, and the quantales of the form $\opens(G)$ are precisely the multiplicative open quantal frames. In section \ref{sec:localbisections} we study a notion of local bisection for open quantal frames that generalizes the corresponding notion for groupoids, and in section \ref{section4} we use this notion and a corresponding action of local bisections on quantales in order to define a weak form of multiplicativity which ensures that the set of local bisections of an open quantal frame has the structure of an inverse semigroup. Finally, in section \ref{sec5}, for such a weakly multiplicative quantale $Q$ we study sufficient (but not necessary) conditions that ensure its multiplicativity. These conditions concern the extent to which $Q$ can be embedded into the inverse quantal frame $\lcc(\sections(Q))$ that arises as the completion of the inverse semigroup $\sections(Q)$ of local bisections of $Q$. We finish by studying the groupoids $G$ whose quantales $\opens(G)$ satisfy the embedding conditions, concluding that for any such groupoid there is an epimorphism of groupoids $J:\cover G\to G$ that provides a canonical ``\'etale cover'' of $G$. This is the case, in particular, for Lie groupoids.

Throughout the paper we shall adopt fairly standard terminology and notation for quantales, locales, groupoids, etc., mostly staying close to \cite{aim}. In particular, we shall often adopt (contrary to what we have done above in this introduction) the common convention of writing $\opens(A)$ for a locale $A$ when it is regarded as an object of the category of frames instead of the category of locales. For instance, using this convention we may write
\[\opens(A\times B)=\opens(A)\otimes\opens(B)\;,\]
where $A\times B$ is the product of the locales $A$ and $B$ and $\opens(A)\otimes\opens(B)$ is their coproduct as frames, which coincides with their tensor product as sup-lattices.

\section{Groupoid quantales}\label{sec:groupoidquantales}

This section is dedicated to establishing the correspondence between open localic groupoids
and multiplicative open quantal frames in a way that directly generalizes the correspondence
between \'{e}tale groupoids and inverse quantal frames.

\paragraph{Inverse quantal frames.}

Let us begin with a brief overview of some of the definitions and results of \cite{aim} concerning \'etale groupoids and inverse quantal frames.
As mentioned in section \ref{introduction}, for any localic \'etale groupoid
\begin{equation}\label{groupoiddiagram}
G\ \ =\ \ \xymatrix{
G_1\times_{G_0}G_1\ar[rr]^-{m}&&G_1\ar@(ru,lu)[]_i\ar@<1.2ex>[rr]^r\ar@<-1.2ex>[rr]_d&&G_0\ar[ll]|u
}
\end{equation}
the sup-lattice $\opens(G_1)$ has the structure of a quantale, denoted by $\opens(G)$, whose multiplication is defined by the following composition:
\[\xymatrix{\opens(G_1)\otimes \opens(G_1)\ar@{->>}[r] &\opens(G_1\times_{G_0}G_1)\ar[r]^-{m_!}& \opens(G_1)}\;.\]
This quantale is involutive with the involution defined by $a^*=i_!(a)$, and it is unital with $e=u_!(1_{G_0})$ --- in other words, the ``open subspace'' $G_0$ is the multiplicative unit of $\opens(G)$. In addition, there is a so-called \emph{stable support}
\[\spp=u_!\circ d_!:\opens(G)\to\opens(G)\;,\]
by which is meant a sup-lattice endomorphism of $\opens(G)$ satisfying the following properties:
\begin{eqnarray}
\spp(a) &\le& e\label{spp1}\;,\\
\spp(a) &\le& aa^*\label{spp2}\;,\\
a &\le& \spp(a) a\label{spp3}\;,\\
\spp(ab) &=& \spp(a\spp(b))\;.\label{spp4}
\end{eqnarray}
Conditions (\ref{spp1})--(\ref{spp3}) define a \emph{support}, and the adjective ``stable'' means that (\ref{spp4}) holds.

An important consequence of these properties is that the restriction of $\spp$ to the lattice $\rs(\opens(G))$ of right-sided elements of $\opens(G)$ defines an order isomorphism $\rs(\opens(G))\to\downsegment e$, whose inverse is defined by multiplication by $1=1_{\opens(G)}$ on the right: $b\mapsto b 1$. In particular, both $\downsegment e$ and $\rs(\opens(G))$ are frames, and we obtain the following order isomorphisms:
\begin{equation}\label{isomorphisms}
\opens(G_0)\cong\downsegment e\cong\rs(\opens(G))\;.
\end{equation}

Finally, the elements $s\in\opens(G)$ such that $ss^*\le e$ and $s^*s\le e$ are called \emph{partial units}. They have an obvious correspondence with the local bisections of $G$, which are the local sections $s:U\to G_1$ of $d$, with $U$ an open sublocale of $G_0$, such that $r\circ s:U\to G_0$ is an open regular monomorphism of locales: a partial unit corresponds to the image of $s$, which is an open sublocale of $G_1$ (\cf\ section \ref{sec:localbisections}).
The set of all the partial units of $\opens(G)$ is denoted by $\ipi(\opens(G))$ and it has the structure of a complete and infinitely distributive inverse semigroup, which we abbreviate to \emph{abstract complete pseudogroup (ACP)}, and it covers $G_1$:
\begin{equation}\V\ipi(\opens(G))=1\;.\label{covercondition}
\end{equation}
In other words, $\opens(G)$ is an instance of the following definition:

\begin{definition}(\cite{aim})
By an \emph{inverse quantal frame} $Q$ is meant a frame which is equipped with the additional structure of a unital involutive quantale (\ie, a unital involutive \emph{quantal frame}) such that $\V\ipi(Q)=1$, and for which there is a (necessarily stable and unique) support.
\end{definition}

Every inverse quantal frame is isomorphic to one of the form $\opens(G)$, for a unique (up to isomorphism) \'etale groupoid $G$. Let us briefly describe a specific construction of an \'etale groupoid $G=\groupoid(Q)$ from an inverse quantal frame $Q$. The locale of units $G_0$ is defined by the condition
\begin{equation}\opens(G_0)=\downsegment e\label{defG0etale}
\end{equation}
and, of course, we put
\begin{equation}\opens(G_1)=Q\;.\label{defG1}
\end{equation}
The involution is given by $i_!(a)=i^*(a)=a^*$,
the domain and range maps $d,r:G_1\to G_0$ are defined by the conditions $d_!(a)=\spp(a)$ and $r_!(a)=\spp(a^*)$, or $d^*(b)=b1$ and $r^*(b)=1b$, and the inclusion of units $u:G_0\to G_1$ is defined by $u_!(b)=b$, or $u^*(a)=a\wedge e$.
Most of what is left has already been described in section \ref{introduction}. In particular, $Q$ is a $\downsegment e$-$\downsegment e$-bimodule under multiplication on both sides, and the multiplication of $Q$ factors (due to associativity) in the category of sup-lattices as
\begin{equation}\xymatrix{Q\otimes Q\ar@{->>}[r]&Q\otimes_{\downsegment e} Q\ar[r]^-\mu& Q}\;.\label{defmu}
\end{equation}
It is then crucial (and nontrivial) that the right adjoint $\mu_*$
preserves joins, a property that is referred to as \emph{multiplicativity} of $Q$. This means that $\mu_*$ is a frame homomorphism, and the multiplication of the groupoid \[m:G_1\times_{G_0} G_1\to G_1\] is defined by the condition $m^*=\mu_*$.

\paragraph{Balanced quantal frames.}

If the localic groupoid $G$ of (\ref{groupoiddiagram}) is open but not \'etale we still have an involutive quantale $\opens(G)$ as above, but this quantale is no longer unital (equivalently, the map $u$ is not open), and we cannot identify $\opens(G_0)$ with a subquantale of $\opens(G)$. However, there is still an isomorphism $\opens(G_0)\cong\rs(\opens(G))$, and this suggests an alternative way of defining $G_0$ in terms of $\opens(G)$ (of course, we could use the left side $\ls(\opens(G))$ instead). We shall use this fact as a motivation for the characterization of the quantales of the form $\opens(G)$ whose study we now begin.

From now on let $Q$ be an arbitrary but fixed involutive quantal frame. We shall denote by $\delta$ the frame inclusion $\rs(Q)\to Q$ and by $\gamma$ the restriction of the involution map $(-)^*:\rs(Q)\to Q$ (another frame homomorphism). Associated to $Q$ there is an obvious involutive localic graph
\begin{equation}\label{prelimgraphdiagram}
G\ \ =\ \ \xymatrix{
G_1\ar@(ru,lu)[]_i\ar@<1.2ex>[rr]^r\ar@<-1.2ex>[rr]_d&&G_0
}\;,
\end{equation}
defined by the conditions
$\opens(G_1)=Q$, $\opens(G_0)=\rs(Q)$, $d^*=\delta$, $r^*=\gamma$, and $i^*(a)=a^*$ for all $a\in Q$. Saying that $G$ is involutive means simply that $i\circ i=\ident$ and $d\circ i=r$ (and $r\circ i=d$).

Regarding $\rs(Q)$ as a subframe of $Q$ (rather than a subquantale),
we define on $Q$ the structure of an $\rs(Q)$-$\rs(Q)$-bimodule whose left and right action are given by, for $a\in Q$ and $z\in\rs(Q)$,
\begin{eqnarray}
z\cdot a &=& a \wedge z\;,\label{leftrsaction}\\
a\cdot z &=& a \wedge z^*\;.\label{rightrsaction}
\end{eqnarray}

\begin{lemma} The frame pushout of $\gamma$ and $\delta$
\[
\xymatrix{
Q\otimes_{\rs(Q)} Q&& Q\ar[ll]_{\iota_2}\\
Q\ar[u]^{\iota_1}&&\rs(Q)\ar[ll]^{\gamma}\ar[u]_{\delta}}
\]
coincides with the tensor product $Q\otimes_{\rs(Q)}Q$ of $Q$ with itself under the $\rs(Q)$-$\rs(Q)$-bimodule structure defined in (\ref{leftrsaction}) and (\ref{rightrsaction}).

\end{lemma} 

\begin{proof}
The frame $Q\otimes_{\rs(Q)} Q$ is a quotient of the coproduct $Q \otimes Q$, which coincides
with the tensor product of sup-lattices. The quotient is defined by the condition
\[ z^* \otimes 1 = \iota_1(\gamma(z)) = \iota_2(\delta(z)) = 1 \otimes z  \]
for $z \in \rs(Q)$.
Stabilizing under meets, we get, for $a,b \in Q$,
\[ (a\wedge z^*) \otimes b = a\otimes (b\wedge z)\;,\]
that is,
\[a\cdot z \otimes b = a\otimes z\cdot b\;,\]
which is the required condition defining the tensor product. \qed
\end{proof}

\begin{definition}\label{def:balanced:ineq}
We say that $Q$ is \emph{balanced}
if 
\[b(a1 \wedge c) =(b \wedge 1a^{*})c\]
for all $a,b,c\in Q$.
\end{definition}
[It suffices to impose $b(a1 \wedge c) \le(b \wedge 1a^{*})c$ for all $a,b,c\in Q$, due to the involution.]

\begin{lemma}
If $Q$ is balanced and $\rs(Q) = Q1$, the quantale multiplication $\mu:Q\otimes Q\to Q$ has the following factorisation in the category of sup-lattices, where we denote by $\pi$ the frame surjection 
$Q\otimes Q \rightarrow Q\otimes_{\rs (Q)} Q$:
\[
\xymatrix{
Q\otimes Q\ar[d]_{\pi}\ar[drr]^{\mu}\\
Q\otimes_{\rs(Q)} Q \ar[rr]_{\mu_0}&& Q}
\]
\end{lemma}

\begin{proof}
First we use the fact that $\rs(Q)=Q1$ in order to write every $z\in\rs(Q)$ in the form $z=c1$ for $c\in Q$.
By definition the multiplication $\mu:Q\otimes Q\to Q$ preserves joins in each variable, and furthermore it is ``middle-linear'' because, since $Q$ is balanced, we have
\[ \mu( a \otimes (c1\cdot b)) = a(c1 \wedge b) = (a \wedge 1c^*)b  = \mu( (a\cdot c1) \otimes b)\]
for all $a,b,c \in Q$.
The factorization follows from the definition of the tensor product. \qed
\end{proof}

Henceforth we shall use the following terminology:
\begin{definition}
If $Q$ is balanced, we refer to the homomorphism
\[\mu_0:Q\otimes_{\rs(Q)} Q\to Q\]
in the above factorization as the \emph{reduced multiplication} of $Q$.
By a \emph{multiplicative quantal frame} is meant a balanced quantal frame such that the right adjoint of the reduced multiplication preserves joins.
\end{definition}

\begin{remark}
Warning: this notation is at odds with the notation that was used above and in \cite{aim} for inverse quantal frames, since the multiplication $\mu_0$, which is now being called reduced, was previously denoted by $\mu$.
\end{remark}

It is immediate that if $Q$ is multiplicative the localic graph of (\ref {prelimgraphdiagram}) is equipped with a multiplication $m$,
\begin{equation}\label{graphdiagram}
G\ \ =\ \ \xymatrix{
G_2\ar[rr]^-{m}&&G_1\ar@(ru,lu)[]_i\ar@<1.2ex>[rr]^r\ar@<-1.2ex>[rr]_d&&G_0
}\;,
\end{equation}
where $G_2$ is the pullback of $d$ and $r$ and $m$ is defined as follows:
\begin{eqnarray}
m^*(a)&=& (\mu_0)_*(a)=\V_{xy\le a}\ x\otimes y\;.
\end{eqnarray}

It is straightforward to verify the following:

\begin{lemma}
Let $Q$ be a multiplicative quantal frame. The multiplication of the graph $G$ is associative and $i$ is an involution for it.
\end{lemma}

\begin{proof}
The proof of associativity is entirely analogous to the proof of associativity in \cite[Th.\ 4.8]{aim}. Saying that $i$ is an involution for $m$ means that it satisfies
\begin{eqnarray}
i\circ m &=& m\circ \chi \label{invol2}
\end{eqnarray}
where $\chi:G_2\to G_2$ is the isomorphism $\langle i\circ\pi_2, i\circ\pi_1\rangle$, whose direct image is
given by $\chi_!(a\otimes b)= b^*\otimes a^*$; and condition (\ref{invol2}) follows from $(i\circ m)_!(a\otimes b)=(ab)^*=b^* a^*=(m\circ\chi)_!(a\otimes b)$. \qed
\end{proof}

In fact, adding the mild condition $\rs(Q)=Q1$ we obtain:

\begin{theorem}\label{semicat}
Let $Q$ be a multiplicative quantal frame such that $\rs(Q)=Q1$. The graph $G$ is an involutive semicategory (\ie, an involutive ``category without units'').
\end{theorem}

\begin{proof}
All there is left to do is to prove that the following diagrams are commutative.
\begin{equation}\label{semicategory}
\xymatrix{
G_2\ar[r]^{\pi_1}\ar[d]_{m}&G_1\ar[d]^d\\
G_1\ar[r]_d&G_0
}\hspace*{2cm} \xymatrix{
G_2\ar[r]^{\pi_2}\ar[d]_{m}&G_1\ar[d]^r\\
G_1\ar[r]_r&G_0
}
\end{equation}
We verify the equation $d\circ m=d\circ\pi_1$ using inverse image homomorphisms. Indeed, for all $z=a1\in \rs(Q)$ we see that $\pi_1^*(d^*(a1))=a1\otimes 1\le m^*(d^*(a1))$ by taking $x=a1$ and $y=1$ in
\[m^*(d^*(a1))=\V_{xy \leq a1} x \otimes y\;,\]
and the converse inequality is proved as follows:
\begin{eqnarray*}
m^*(d^*(a1)) &=& \V_{xy \leq a1} x \otimes y = \V_{xy1 \leq a1} x \otimes y1 = \V_{xy1 \leq a1} x \otimes (y1 \wedge 1)\\
& =& \V_{xy1 \leq a1} (x \wedge 1y^{*}) \otimes 1 \le \V_{xy1 \leq a1} (x \wedge 1y^{*})1 \otimes 1\\
& =& \V_{xy1 \leq a1} x(1 \wedge y1) \otimes 1 = \V_{xy1 \leq a1} xy1 \otimes 1  \\
&\leq& a1 \otimes 1=\pi_1^*(d^*(a1)) \;.
\end{eqnarray*}
The condition $r\circ m=r\circ\pi_2$ is proved analogously. \qed
\end{proof}

Note that this proof did not use multiplicativity so that in fact the following slightly more general fact holds:

\begin{corollary}\label{cor:framehoms}
Let $Q$ be a balanced quantal frame satisfying $\rs(Q)=Q1$. Both $(\mu_0)_*\circ d^*$ and $(\mu_0)_*\circ r^*$ are frame homomorphisms and we have the equalities
\begin{equation}
(\mu_0)_*\circ d^* = \pi_1^*\circ d^*\hspace*{2cm}(\mu_0)_*\circ r^* = \pi_2^*\circ r^*\;.
\end{equation}
\end{corollary}

\paragraph{Open quantal frames.}

We have seen that a multiplicative quantal frame $Q$ (at least one satisfying the mild condition $\rs(Q)=Q1$) defines much of the structure which is necessary in order to obtain a localic groupoid $G$, but a crucial ingredient is missing, namely the inclusion of units $u:G_0\to G_1$. We address this now.

Let $Q$ be a balanced quantal frame and let us consider the map $\upsilon :  Q \rightarrow  Q$ given by
\[\upsilon(a)  = \V_{xy^{*}\leq a} x\wedge y \;. \]
As we shall see later, this is intended to play the role of $u^*:G_1\to G_0$, but for now we begin with a useful technical observation:

\begin{lemma}\label{xxstarvsxystar}
For all $a\in Q$ we have:
\begin{eqnarray}
\upsilon(a) &=& \V_{xx^{*}\leq a} x\;,\label{xxstarvsxystar1}\\
\upsilon(a^*) &=& \upsilon(a)\;.\label{xxstarvsxystar2}
\end{eqnarray}
\end{lemma}
\begin{proof}
We first show that $\upsilon(a) \leq \V_{xx^{*}\leq a}x $.
Indeed $\V_{xy^{*}\leq a} x \wedge y  \leq  \V_{xx^{*}\leq a}x$,
for $ xy^{*} \leq a$ implies that $ (x \wedge y)(x \wedge y)^{*} \leq a$
and hence $x \wedge y \leq \V_{xx^{*}\leq a}x$.
Now, for the other inequality, we have
\[ \V_{xx^{*}\leq a}x \leq  \V_{xx^{*}\leq a}x \wedge x \leq \V_{xy^{*}\leq a} x \wedge y\;,\]
which proves (\ref{xxstarvsxystar1}). And (\ref{xxstarvsxystar2}) is an obvious corollary, since $xx^*\le a$ is equivalent to $xx^*\le a^*$. \qed
\end{proof}

Note that we have:

\begin{lemma}
$\upsilon$ preserves finite meets.
\end{lemma}

\begin{proof}
It is obvious, by the definition of $\upsilon$, that $\upsilon$ is monotone, and this gives us the inequality $\upsilon(a\wedge b)\le\upsilon(a)\wedge\upsilon(b)$. 
For the other inequality,
\begin{eqnarray*}
\upsilon(a) \wedge \upsilon(b) &=& \left(\V_{xx^{*}\leq a} x\right)\wedge\left(\V_{yy^{*}\leq b}
y\right)\\
 &=&\V_{\scriptsize\begin{array}{rcl}xx^{*} &\leq& a\\ yy^{*} &\leq& b\end{array}} x \wedge y\;,
\end{eqnarray*}
but since
for $x$ and $y$ such that $xx^{*} \leq a$ and $yy^{*} \leq b$ we have
\[(x \wedge y)(x \wedge y)^* = (x \wedge y)(x^* \wedge y^*) \leq xx^* \wedge yy^* \leq a \wedge b\]
we get, using \ref{xxstarvsxystar},
$\upsilon(a) \wedge \upsilon(b) \leq \upsilon(a \wedge b)$. \qed
\end{proof}
 
Also, obviously:

\begin{lemma}\label{obvious}
If $Q$ is multiplicative then
 $\upsilon = [\ident_Q,i^*] \circ (\mu_0)_{*}$.
\end{lemma}

Note that, although for multiplicative $Q$ the map $\upsilon$ is a frame homomorphism, it does not yet produce the desired splitting $u:G_0\to G_1$ of $d$ and $r$, to begin with because the image of $\upsilon$ does not necessarily lie in $\rs(Q)$. We include this requirement in the following definition:

\begin{definition}
A balanced quantal frame $Q$ is called \emph{semiopen} if it satisfies
\begin{equation}\label{axiomR}
\upsilon(Q)\subset\rs(Q)
\end{equation}
and, in addition, the following property holds for all $a\in Q$:
\begin{equation}\label{axiomU}
\V_{xx^{*}x\leq a} x = a\;.
\end{equation}
If, furthermore, the following condition is satisfied for all $a,b,c\in Q$,
\begin{equation}\label{axiomO}
(a1 \wedge b)c = a1 \wedge bc\;,
\end{equation}
$Q$ will be called \emph{open}.
Henceforth we shall usually refer to conditions (\ref{axiomR}), (\ref{axiomU}), and (\ref{axiomO}), respectively, as properties \emph{(R)}, \emph{(U)}, and \emph{(O)}. We shall also say that an involutive quantal frame satisfies \emph{(B)} if it is balanced.
\end{definition}

\begin{remark}\label{RQeqQ1}
Note that (U) implies that we
get, for $q \in Q$, 
\[q \leq qq^{*}q \leq q1 \]
and
\[ q \leq qq^*q \leq 1q\;.\]
To see this, take $a = qq^{*}q$ in (U). We obtain, making $x = q$, 
$q \leq qq^{*}q$.
Also,
\[\rs(Q) = Q1\;,\]
for if $z \in \rs(Q)$ then $z1 \leq z$ and, since  $z1 \geq z$, we
obtain $z= z1$. In particular, we remark that an involutive quantal frame satisfying (U) is necessarily a Gelfand quantale because for $z \in\rs(Q)$ we have \[z =  z1 \leq zz^*z1 = zz^*z \leq z1 =z\;.\]
\end{remark}

\begin{example} The axiom (R) is independent of (B) and (U), as the following 
example shows. Consider the commutative involutive quantal frame
 $Q = \mathcal{P}(X)$ for $X = \{a,b\}$ with the trivial involution and multiplication 
generated by \[\{a\}^{2} = \{a\}, \{b\}^{2} = \{b\}, \{a\}\{b\} = \{b\}\{a\} = X \;.\]
It is easy to verify that $Q$ satisfies  (B) (and (O)) and (U) but not
(R) because
\[\left(\bigcup_{WW \subset \{a\}}W\right)X = \{a\}X = X \neq \{a\} = \bigcup_{WW \subset \{a\}}W\;.\]
\end{example}

\begin{example}
Axiom (U) is in turn independent, for consider again the quantale $Q$ of the previous example
but with involution given by
\[i(\{a\}) = \{b\}, i(\{b\}) = \{a\}\;. \]
Then $Q$ continues to satisfy (B) (and (O)), and also (R) because in this case
\[\upsilon(\{a\}) = \upsilon (\{b\}) = \emptyset, \upsilon(X) = X\;, \]
but not (U) because
\[\bigcup_{WW^{*}W \subset \{a\}}W = \emptyset \neq \{a\}\;. \]
\end{example}

\begin{example} Finally, (U) and (O) can be separated even if both (B) and (R) hold because, as we shall see, if $G$ is a localic groupoid whose multiplication map is semiopen but not open the quantale $\opens(G)$ is semiopen but not open, whereas if $G$ is an open groupoid the quantale $\opens(G)$ is open.
\end{example}

\begin{lemma}\label{domain} Let $Q$ satisfy (U). Then $\upsilon(z) = z$ for all $z\in\rs(Q)$. 
\end{lemma}

\begin{proof}
Let $z\in \rs(Q)$.
We have $\upsilon(z) \leq z$ because if $xx^*\leq z1=z$ then
$x \leq xx^*1 \leq z11 = z$ (\cf\ \ref{RQeqQ1}). Also we have $\upsilon(z) \geq z$ because
\[ zz^* = z1z^* \leq z1\;. \qed\]
\end{proof}

\begin{remark} Hence, for $Q$ semiopen $\upsilon$ is surjective onto $\rs(Q)$, since $\rs(Q) = Q1$.
\end{remark}

\begin{lemma}\label{range} For $Q$ satisfying (U) we have $\upsilon \circ \gamma = \ident_{\rs(Q)}$.
\end{lemma}

\begin{proof}
Let $z \in \rs(Q)$. Then $ \upsilon (\gamma(z)) = \upsilon(1z^{*}) \geq  z1 = \upsilon(z1) $ because
$zz^* = z1z^*  \leq 1z^*$.
On the other hand, $ \upsilon(1z^{*}) \leq z1 = \upsilon(z1)$ because
$ xx^* \leq 1z^*$ implies that $xx^* \leq z1$ and $xx^*1 \leq z1$. Hence, $x \leq xx^*x \leq xx^*1 \leq z1. \qed$
\end{proof}
 
The following lemma will be important:
 \begin{lemma}\label{U-lemma}
 (U) holds if and only if for all $a\in Q$ we have
 \[\V_{xy \leq a} \; \V_{pq^{*} \leq x} p \wedge q\wedge y = \V_{xy \leq a} \upsilon(x) \wedge y = a\;. \]
  \end{lemma}
[Later we shall sometimes use the above alternative form for the
   axiom (U).]
   
   \begin{proof}
    Indeed,
    $ xy \leq a$ and $pq^{*} \leq x $ implies that \[(p\wedge q \wedge y)(p \wedge q \wedge y)^{*}
    (p \wedge q \wedge y) \leq pq^{*}y \leq xy \leq a\] and, hence,
       \[\V_{xy \leq a}\;\V_{pq^{*} \leq x} p \wedge q\wedge y \leq \V_{xx^{*}x \leq a} x\;. \]
      For the other inequality
     let $w$ be such that $ww^{*}w \leq a$. Then the particular instance
     $y = w$, $x = ww^{*}$, $p=q = w$ gives us
     \[w=w\wedge w\wedge w \leq  \V_{xy \leq a}\; \V_{pq^{*} \leq x} p \wedge q\wedge y\;,\]
and taking the supremum of such $w$'s we get:
      \[\V_{ww^{*}w \leq a} w \leq \V_{xy \leq a}\; \V_{pq^{*} \leq x} p \wedge q\wedge y\;.\]
      Hence,
      \[\V_{xx^{*}x \leq a} x = \V_{xy \leq a}\; \V_{pq^{*} \leq x} p \wedge q\wedge y\;,\]
      and the lemma follows. \qed
      \end{proof}

\begin{lemma}\label{upsilon}
Let $Q$ be an open quantal frame. Then for all $a,b\in Q$ we have
\[ \upsilon(a) \leq a1 \leq \upsilon(aa^{*}) \]
and
\[ \upsilon(a) \wedge b \leq ab\;.\]
\end{lemma}

\begin{proof}
Since $a \leq \upsilon(aa^*)$ and $\upsilon(aa^*) \in \rs(Q)$ we obtain $a1 \leq \upsilon(aa^*)$.
We have $\upsilon(a) \wedge b \leq \V_{pq \leq ab} \upsilon(p) \wedge q = ab$ using (U).
Hence, we also obtain $\upsilon(a) = \upsilon(a) \wedge 1 \leq a1$. \qed
\end{proof}

\begin{lemma}\label{semiopenlemma}
If $Q$ is a semiopen quantal frame the maps $d$ and $r$ of the associated graph (\ref{groupoiddiagram}),
\[
G\ \ =\ \ \xymatrix{
G_1\ar@(ru,lu)[]_i\ar@<1.2ex>[rr]^r\ar@<-1.2ex>[rr]_d&&G_0
}\;,
\]
are semiopen (\ie, $d^*$ and $r^*$ have left adjoints $d_!$ and $r_!$) and the direct image homomorphisms are defined by, for all $a\in Q$,
\begin{eqnarray*}
d_!(a) &=& a1\;,\\
r_!(a) &=& a^*1\;.
\end{eqnarray*}
In addition, if $Q$ is open then so are $d$ and $r$.
\end{lemma}
\begin{proof}
Let $Q$ be semiopen and let $z\in\rs(Q)$. We have
$d_! \circ d^* (z) = z1 = z$ and
$d^* \circ d_!(a) = a1 \geq a$, whence $d_!\dashv d^*$.
Also, $r_!\dashv r^*$ because we have
$r_! \circ r^*(z) = z^{**}1 = z1 = z$
and
$r^* \circ r_!(a) = (a^{*}1)^{*} = 1a \geq a$. It is easy to see that if in addition $Q$ satisfies (O) then
 $d$ and $r$ are open, for the 
Frobenius reciprocity condition holds: for all $a\in Q$ and all $z\in\rs(Q)$ we have\[d_{!}(d^*(z)\wedge a) =(z\wedge a)1 =(z1\wedge a)1= z1 \wedge a1 = z\wedge d_{!}(a)\;. \qed\]
\end{proof}

\paragraph{Quantal groupoids.}

Recall that by a \emph{quantal groupoid} \cite{aim} is meant a localic groupoid $G$ whose multiplication map is semiopen --- in other words, such that the associated quantale $\opens(G)$ is defined \cite[Th.\ 5.2]{aim}.

Now we shall see that every multiplicative semiopen (resp.\ open) quantal frame $Q$ has an associated quantal (resp.\ open) groupoid $\groupoid(Q)$ and, conversely, that the associated quantale $\opens(G)$ of a quantal (resp.\ open) groupoid $G$ is necessarily a multiplicative semiopen (resp.\ open) quantal frame.

\begin{theorem}
Let $Q$ be a multiplicative semiopen quantal frame, and let $G$ be its associated involutive localic graph, as in (\ref{graphdiagram}),
\[
\xymatrix{
G_2\ar[rr]^{m}&&G_1\ar@(ru,lu)[]_i\ar@<1.2ex>[rr]^r\ar@<-1.2ex>[rr]_d&&G_0\ar[ll]|u
}\;,
\]
where the map $u:G_0\to G_1$ is defined by $u^*=\upsilon$. Then $G$ is a quantal groupoid. Furthermore, if $Q$ is open $G$ is an open groupoid.
\end{theorem}

\begin{proof}
First, $u$ is well defined because multiplicativity makes $\upsilon$ equal the frame homomorphism $[\ident_Q,i^*] \circ (\mu_0)_{*}$ (\cf\ \ref{obvious}), and its domain is indeed $G_0$ because $Q$ is semiopen and thus the image of $\upsilon$ is $\rs(Q)$. Since semiopen quantal frames satisfy $\rs(Q)=Q1$ (\cf\ \ref{RQeqQ1}) we conclude, by \ref{semicat}, that $G$ is an involutive semicategory, and the remaining properties to be checked are those that relate to the unit map $u:G_0\to G_1$. First, we note that the following properties of a reflexive graph hold:
\begin{itemize}
\item $d\circ u=\ident_{G_0}$ holds due to \ref{domain};
\item $r\circ u =\ident_{G_0}$ holds due to \ref{range}.
\end{itemize}
Now we prove the unit laws of an internal category, as illustrated by the following commutative diagram:
\begin{equation}
\vcenter{\xymatrix{
G_0\times_{G_0}G_1\ar[rr]^{u\times\ident}&&G_1\ptimes{G_0}G_1\ar[d]_m&&
G_1\ptimes{G_0}G_0\ar[ll]_{\ident\times u}\\
G_1\ar[u]^-{\langle d,\ident\rangle}\ar@{=}[rr]&&G_1&&G_1\ar[u]_-{\langle\ident,r\rangle}\ar@{=}[ll]
}}
\end{equation}
The commutativity of the left hand square can be proved in terms of inverse images. For all $a\in Q$ we have, using \ref{U-lemma}:
\begin{eqnarray*}
[d^*,\ident]\circ(u^*\otimes \ident)\circ{m}^{*}(a) &=& [d^*,\ident]\circ(u^*\otimes \ident)\left(\V_{xy \leq a} x \otimes y\right)\\
&=& \V_{xy \leq a} u^*(x) \wedge y=a\;.
\end{eqnarray*}
The commutativity of the right hand square follows from the left one using the involution laws $d\circ i=r$ and $i\circ i=\ident$ together with (\ref{invol2}) and $i\circ u=u$ [the latter is a consequence of (\ref{xxstarvsxystar2})]:
\begin{eqnarray*}
m\circ(\ident\times u)\circ\langle\ident,r\rangle&=&m\circ\langle\ident,u\circ r\rangle=m\circ\langle i\circ i,i\circ u\circ d\circ i\rangle\\
&=& m\circ( i\times i) \circ\langle\ident, u\circ d\rangle\circ i = m\circ\chi\circ\langle u\circ d,\ident\rangle\circ i\\
&=& i\circ m\circ \langle u\circ d,\ident\rangle\circ i=i\circ\ident\circ i =\ident\;.
\end{eqnarray*}
Finally, in order to see that $G$ is a groupoid we prove that the involution $i$ satisfies the inverse laws described by the commutativity of the following diagram:
\begin{equation}
\vcenter{\xymatrix{
G_1\ar[rr]^-{\langle \ident,i \rangle}\ar[d]_-{d}&&G_2\ar[d]_-m&&
G_1\ar[ll]_-{\langle i, \ident \rangle}\ar[d]^-{r}\\
G_0\ar[rr]_-{u}&&G_1&&G_0\ar[ll]^-{u}\;.
}}
\end{equation}
Again it is straightforward to see that the inverse laws make the commutativity of the two squares equivalent, so we prove only the commutativity of the left square, using inverse image homomorphisms: for all $a\in Q$ we have
\begin{eqnarray*} [\ident, i^*]\circ m^{*}(a) &=& [\ident, i]\left( \V_{xy \leq a} x \otimes y\right)\\
&=& \V_{ xy \leq a} x \wedge y^{*} = \V_{ xy^* \leq a} x \wedge y\\ &=& u^*(a) = d^*(u^*(a))\;,
\end{eqnarray*}
where the last step follows from the condition that $u^*(a)$ is right-sided and $d^*$ is just the inclusion of $\rs(Q)$ into $Q$.
Hence, $G$ is a groupoid, and it is semiopen because $m^*$ has the left adjoint $\mu_0$. If $Q$ is open then $d$ is open, by \ref{semiopenlemma}, and thus $G$ is open. \qed
\end{proof}

\begin{definition}\label{def:associated_groupoid}
Given a multiplicative semiopen quantal frame $Q$, we denote its associated quantal groupoid by $\groupoid(Q)$.
\end{definition}

Now let us show that from quantal (resp.\ open) groupoids one obtains semiopen (resp.\ open) quantal frames.
  
\begin{theorem}\label{groupoidquantale}
Let $G$ be a quantal groupoid. Then its associated quantale $\opens(G)$ is a semiopen quantal frame, and it is multiplicative. Furthermore, if $G$ is open so is $\opens(G)$.
\end{theorem}

\begin{proof}
From \cite[Lemma 5.4]{aim} we obtain $\upsilon = d^*\circ u^*$
and from \cite[Lemma 5.3]{aim} it follows that $\upsilon(a)$ is right-sided for all $a\in\opens(G_0)$. Hence, axiom (R) holds. Furthermore, the unit law $m\circ \langle u\circ d,\ident\rangle=\ident$ of the groupoid gives us
\[a =[\upsilon,\ident] \circ m^{*}(a)
= [\upsilon,\ident] \left(\V_{xy \leq a} x \otimes y\right)= \V_{xy \leq a}\upsilon(x) \wedge y\:,\]
and thus, by \ref{U-lemma}, (U) holds. Now in order to see that $\opens(G)$ is semiopen it we must show that $\opens(G)$ satisfies (B). The frame $\opens(G_1\times_{G_0} G_1)$ is the pushout of $d^*$ and $r^*$ and thus it satisfies the equation
\[(a\wedge d_{*}(c)) \otimes b  = a \otimes (r_{*}(c) \wedge b)\]
for all $a,b\in\opens(G_1)$ and all $c\in\opens(G_0)$.
In particular, taking $c=u^*(z1)$ for some $z\in\opens(G)$ and recalling from \ref{domain} that $\upsilon(z1)=z1$ we obtain $d^*(u^*(z1))=z1$ and $r^*(u^*(z1))=1z^*$ and, hence, (B) holds:
\[(a\wedge z1)b=m_!(( a\wedge z1) \otimes b ) = m_!(a \otimes (1z^{*} \wedge b))=a(1z^*\wedge b)\;.\]
The multiplicativity of $\opens(G)$ is obvious, of course, because $(\mu_0)_*$ is the frame homomorphism $m^*$.
Now suppose furthermore that $G$ is an open groupoid and let us show that $\opens(G)$ is an open quantal frame; that is, we must show that $(a1\wedge b)c$ equals $a1\wedge bc$ for all $a,b,c\in \opens(G)$. First, since $\pi_1^*=q\otimes 1$ for all $q\in\opens(G)$, we have
\begin{eqnarray}
(a1 \wedge b)c &=& m_!((a1 \wedge b)\otimes c)=m_!((a1 \wedge b)\otimes (1\wedge c))\\
&=&m_!(a1\otimes 1\wedge b\otimes c)=m_!(\pi_1^*(a1)\wedge b\otimes c)
\;.\label{intermediate}
\end{eqnarray}
Now notice that the following is a pullback diagram because $G$ is a groupoid rather than just a category:
\begin{equation}\vcenter{\xymatrix{G_2\ar[r]^{\pi_1}\ar[d]_m&G_1\ar[d]^d\\
G_1\ar[r]_d&G_0}}\;.\label{BCdiagram}
\end{equation}
Hence, $m$ is open and the Beck--Chevalley condition gives us
\[m_! \circ \pi_1^*= d^*\circ d_!\;.\]
Equivalently, since $\pi_1^*(a)=a\otimes 1$, this states that
\begin{equation}\label{eqa1}
a1=d^*(d_!(a))
\end{equation}
for all $a\in\opens(G)$.
Then by (\ref{BCdiagram}) we get
\[\pi_1^*(a1)=\pi_1^*(d^*(d_!(a)))=m^*(d^*(d_!(a)))=m^*(a1)\;,\]
and thus the right-hand side of (\ref{intermediate}) equals
\begin{eqnarray*}
m_!(m^*(a1)\wedge b\otimes c)=a1\wedge m_!(b\otimes c)=a1\wedge bc
\;,
\end{eqnarray*}
where the first equality is the Frobenius reciprocity condition for $m$. \qed
\end{proof}

\begin{corollary}\label{inverseimpliesopen}
Inverse quantal frames are necessarily open quantal frames.
\end{corollary}

It is straightforward to see that $\opens$ and $\groupoid$ establish a bijective correspondence, up to isomorphisms, between quantal groupoids and multiplicative semiopen quantal frames, and between open groupoids and multiplicative open quantal frames. This follows from the following result, whose proof we omit:

\begin{theorem} 
$\groupoid(\mathcal{Q}(G))\cong G$ and $\mathcal{Q}(\groupoid(Q))\cong Q$ for
any localic quantal groupoid $G$ and multiplicative semiopen quantal frame $Q$.
\end{theorem}

\paragraph{Inverse quantal frames revisited.}

In order to conclude this section we shall show that, regardless of multiplicativity, open quantal frames are
good non-unital generalizations of inverse quantal frames:

\begin{theorem}
The class of unital open quantal frames coincides with the class of inverse quantal frames.
\end{theorem}

\begin{proof}
Inverse quantal frames are necessarily open, as already stated in \ref{inverseimpliesopen}. In order to see the converse, consider an open quantal frame $Q$ with unit $e$.
Since $a \leq a1$, using (O) we obtain
$a = a1 \wedge a = a1 \wedge ea = (a1 \wedge e )a$. Hence, the sup-lattice endomorphism $\spp$ defined on $Q$ by $\spp(a)=a1\wedge e$ satisfies two of the axioms, (\ref{spp1}) and (\ref{spp3}), of a support:
\[\spp(a)\le e\hspace*{1cm}\textrm{and}\hspace*{1cm}\spp(a)a\le a\;.\]
Now note that from the definition of $\upsilon$ we obtain (make $x=y=a$)
\[\upsilon(aa^*)=\V_{xy\le aa^*}x\wedge y\ge a\wedge a=a\;,\]
and thus $a1\le\upsilon(aa^*)$ because by (R) $\upsilon(aa^*)$ is right-sided. Hence, using \ref{U-lemma} we obtain (make $x=aa^*$ and $y=e$)
\[aa^*=\V_{xy\le aa^*}\upsilon(x)\wedge y\ge\upsilon(aa^*)\wedge e\ge a1\wedge e\;,\]
which gives us the remaining axiom, (\ref{spp2}), of a support:
\[\spp(a)\le aa^*\;.\]
To conclude, in order to show that $Q$ is an inverse quantal frame we need only show that $ 1 = \V \mathcal{I}(Q)$. Due to the involution and the fact that $\upsilon(a)$ is right-sided we see, using \cite[Lemma 4.18]{aim}, that this condition is equivalent to
$a \wedge e \leq \upsilon(a)$. And this is easily seen to be true if we make $x=a$ and $y=e$ in the definition of $\upsilon$. \qed
\end{proof} 

\section{Local bisections}\label{sec:localbisections}

As we have mentioned in the introduction, there is loss of information, in general, in the passage from an open groupoid $G$ to its inverse semigroup $\sections(G)$ of local bisections, whereas, as we have seen in the previous section, the quantale $\opens(G)$ allows us to fully recover the groupoid $G$. However, the characterization of the quantales of the form $\opens(G)$ relies on the cumbersome multiplicativity axiom that in the case of \'etale groupoids is not needed, and this prompts us into seeking a friendlier replacement, or at least approximation, for this axiom. As we shall see, $\sections(G)$ plays an important role in this, and hence we begin by examining the local bisections of open localic groupoids and the extent to which they can be defined for arbitrary open quantal frames.

\paragraph{Local bisections of open quantal frames.}

Throughout this section, $Q$ denotes an arbitrary but fixed open quantal frame, and we retain the notation $d$, $r$, and $i$ of the previous section for the structure maps of its associated involutive localic graph, and in those cases where $\upsilon$ is preserves joins we write $u$ for the map of locales defined by $u^*=\upsilon$. We recall that this happens, for instance, when
$Q$ is assumed to be multiplicative, in which case we denote the multiplication of the associated groupoid by $m$.

We shall often denote the elements of $\rs(Q)$ by roman capitals $U$, $V$, etc., thinking of them metaphorically as the open sets of a space $G_0$. We shall also write $\ds U$ for $\downsegment U\cap\rs(Q)$ (both for the object in the category of locales and in the category of frames --- that is, without using the $\opens$ notation), and we write
$k_U:\ds U\to\rs(Q)$ for the inclusion of the open sublocale $\ds U$ into $\rs(Q)$, for any $U\in\rs(Q)$. Hence, $k_U^*(V) = V\wedge U$ for all $V\in\rs(Q)$.

\begin{definition}
By a \emph{local bisection} of $Q$ is meant a pair $(U,s)$, where $U\in\rs(Q)$ and
\[s:\ds U\to Q\]
is a map of locales such that:
\begin{enumerate}
\item $d\circ s = k_U$ ($s$ is a local section of $d$);
\item $r\circ s$ is an open regular monomorphism of locales.
\end{enumerate}
\end{definition}

The second condition is equivalent to imposing that the standard (epi,regular mono)-factorisation of $r\circ s$ is given by
\[r\circ s = k_V\circ\alpha\]
for an isomorphism of locales $\alpha:\ds U\to\ds V$, where $V\in\rs(Q)$ is the image of $r\circ s$ in $\rs(Q)$.
Then the map $t=s\circ\alpha^{-1}:\ds V\to Q$ is a local section of $r$, for
\[r\circ s\circ\alpha^{-1}=k_V\circ\alpha\circ\alpha^{-1}=k_V\;.\]

Of course, local bisections could have equally been defined in terms of $t$ rather than $s$.
We shall use the following terminology, where the notation is the same as above:

\begin{definition}
Let $\sigma=(U,s)$ be a local bisection of $Q$.
\begin{itemize}
\item $U$ is the \emph{domain} of $\sigma$;
\item $V$ is the \emph{codomain} of $\sigma$;
\item $s$ is the \emph{$d$-section} of $\sigma$, or the \emph{domain section};
\item $t$ is the \emph{$r$-section} of $\sigma$, or the \emph{codomain section};
\item $\alpha$ is the \emph{action} of $\sigma$ (on $\rs(Q)$).
\end{itemize}
When there is ambiguity we shall denote $U$, $s$, $\alpha$, etc., by
$U_\sigma$, $s_\sigma$, $\alpha_\sigma$, etc.
\end{definition}

It will be useful to keep in mind straightforward formulas such as the following.

\begin{lemma} Let $(s,U)$ be a local bisection of $Q$. Then we have
\begin{enumerate}
\item $(d\circ s)^*(a) = k_U^*(a) = a \wedge U$ for all $a \in \rs(Q)$
\item $(r\circ s)^*(a) = (r\circ s)^*(a\wedge V) = \alpha^*(a\wedge V)$ for all $a \in \rs(Q) $
\item $\alpha^*(a) = s^*(r^*(a))$ for all $a\in\ds V$
\item $U=(r\circ s)^*(V)=s^*(r^*(V))$
\item $U=(r\circ s)^*(1)=s^*(r^*(1))$
\item $U=\alpha^*(V)$
\item $V=\alpha_!(U)$
\item $V=(r\circ s)_!(U)$.
\end{enumerate}

\end{lemma}

\begin{proof}
Property 1 follows from the condition of the definition
\[d \circ s = k_U\]
and 2 follows from
\[r \circ s = k_V \circ \alpha\]
because  $(r\circ s)^*(a\wedge V) =  \alpha^*({k_V}^{*}(a\wedge V)) =  \alpha^*(a\wedge V)$.
Property 3 follows directly from 2. Properties
4--8 follow from 2 and the fact that $\alpha$ is an isomorphism of locales between
$\ds U$ and $\ds V$. \qed
\end{proof}

Also we have their ``codomain-duals'', which we state without proof:

\begin{lemma} Let $(s,U)$ be a local bisection of $Q$. Then we have
\begin{enumerate}
\item $(r\circ t)^*(a) = k_V^*(a) = a\wedge V$ for all $a \in \rs(Q)$
\item $(d\circ t)^*(a) = (d\circ t)^*(a\wedge U) = {(\alpha^{-1})}^*(a\wedge U) = \alpha_!(a\wedge U)$ for all $a \in \rs(Q)$
\item $\alpha_!(a) = t^*(d^*(a))$ for all $a\in \ds U$
\item $V=(d\circ t)^*(U)=t^*(d^*(U))$
\item $V=(d\circ t)^*(1)=t^*(d^*(1))$
\item $U=(d\circ t)_!(V)$.
\end{enumerate}
\end{lemma}

Taking into account the specific formulas $d^*(a)=a$ and $r^*(a)=a^*$ for $a\in\rs(Q)$ we further obtain:

\begin{lemma} Let $(s,U)$ be a local bisection of $Q$. Then we have the following
\begin{enumerate}
\item $s^*(a) = k_U^*(a) = a\wedge U$ for all $a\in\rs(Q)$
\item $U=s^*(V^*)$
\item $U=s^*(1)$
\item $s^*(x) = s^*(x)\wedge U = s^*(x\wedge V^*)$ for all $x\in Q$
\item $\alpha^*(a) = s^*(a^*)$ for all $a\in\ds V$
\item $t^*(a^*) = k_V^*(a) = a\wedge V$ for all $a\in\rs(Q)$
\item $V=t^*(U)$
\item $V=t^*(1)$
\item $t^*(x) = t^*(x)\wedge V = t^*(x\wedge U)$ for all $x\in Q$
\item $\alpha_!(a) = t^*(a)$ for all $a\in \ds U$.
\end{enumerate}
\end{lemma}

\begin{lemma}
If $\sigma=(U,s)$ is a local bisection then so is $\sigma^{-1}=(V,i\circ t)$, and we have $\alpha_{\sigma^{-1}}=\alpha_\sigma^{-1}$. Hence, in terms of frame homomorphisms we have
\[s_{\sigma^{-1}}^*(a) = t_\sigma^*(a^*)\]
for all $a\in Q$.
\end{lemma}

\begin{proof}

We have $d \circ i \circ t = r \circ t = k_V$.
Also $r \circ i \circ t = d \circ t =  d \circ s \circ \alpha^{-1} = k_U \circ \alpha^{-1}$.
Hence $r \circ i \circ t$ is a regular monomorphism of locales, since $\alpha^{-1}$ is
an isomorphism. \qed
\end{proof}

We shall write $\sigma^{-1}$ with the above meaning from here on.

\begin{example}\label{globalbisec}
In those cases where $\upsilon$ preserves joins an obvious example of local bisection is $\varepsilon = (u,1)$. 
Then we have $t_\varepsilon=s_\varepsilon=u$ and $\alpha_\varepsilon =\ident_{\rs(Q)}$, and the inverse $\varepsilon^{-1}$ coincides with $\varepsilon$.
\end{example}

\begin{definition}
We shall denote by $\sections(Q)$ the set of local bisections of $Q$. (Later we shall give conditions for this to be an inverse semigroup with 
inverse 
operation $(-)^{-1}$ and multiplicative unit $\varepsilon$.)
\end{definition}

\paragraph{Local bisections of open groupoids.}

We shall continue to denote by $Q$ an arbitrary but fixed open quantal frame, but now we further require it to be multiplicative. As we shall see, in this case
 there is also a multiplication of local bisections whose geometric meaning is the usual one for groupoids, namely 
the ``arrows in the image of $s_\sigma$ are composed with the arrows in the image of $s_\tau$''.

\begin{definition}\label{def:product}
Let $\sigma$ and $\tau$ be local bisections of $Q$. We define their \emph{product}
$\sigma\tau=(U_{\sigma\tau},s_{\sigma\tau})$ as follows:
\begin{enumerate}
\item $U_{\sigma\tau} = \alpha_\sigma^*(U_\tau\wedge V_\sigma) = s_\sigma^*(r^*(U_\tau))=s_\sigma^*(U_\tau^*)$
\item $s_{\sigma\tau} = m\circ\langle  s_\sigma\circ \iota, s_\tau\circ\beta \rangle$, where 
$\iota:\ds{ U_{\sigma\tau}}\to\ds{U_\sigma}$ is the open inclusion of the locale $\ds{U_{\sigma\tau}}$ into 
$\ds{U_\sigma}$ and
$\beta:\ds {U_{\sigma\tau}}\to \ds{U_\tau}$
is the pullback of $k_{V_{\sigma}} \circ \alpha_{\sigma} = r \circ s_\sigma$ along  $k_{U_\tau}$:
\[
\xymatrix{
\ds{U_{\sigma\tau}}\ar[rr]^{\beta}\ar@{_(->}[d]_{\iota}
&&\ds{U_\tau}\ar@{_(->}[d]_{k_{U_\tau}}\\
\ds{U_\sigma}\ar[rr]_{r\circ s_\sigma}&&\rs(Q)\;.
}
\]
\end{enumerate}
\end{definition}

We remark that the pairing in the definition of  $s_{\sigma\tau}$ arises from the definition of $Q\otimes_{\rs(Q)} Q$ as a pullback of $d$ and $r$ in the category of locales; that is, the pairing is well defined as a map
$\ds{U_{\sigma\tau}}\to Q\otimes_{\rs(Q)} Q$ because 
\[r\circ s_\sigma \circ \iota = k_{U_{\tau}}\circ \beta =  d \circ s_\tau \circ \beta\;.\]

\begin{lemma}\label{prop:localbisec}
$(U_{\sigma\tau},s_{\sigma\tau})$ in the above definition is a local bisection.
\end{lemma}

\begin{proof}
We verify that $r\circ s_{\sigma\tau}$ is an open regular monomorphism. Since $r\circ m=r\circ\pi_2$, we have
\begin{eqnarray*}
r\circ s_{\sigma\tau}&=&r\circ m\circ\langle  s_\sigma\circ \iota, s_\tau\circ\beta \rangle\\
&=& r\circ \pi_2\circ\langle  s_\sigma\circ \iota, s_\tau\circ\beta\rangle\\
&=& r\circ s_\tau\circ\beta\;.
\end{eqnarray*}
Open regular monomorphisms of locales are stable under pullback and thus $\beta$ is an open regular monomorphism 
because $r\circ s_\sigma$ is. Moreover, $r\circ s_\tau$ is an open regular monomorphism and thus $r\circ s_{\sigma\tau}$ is the composition of two open regular monomorphisms. \qed
\end{proof}

\begin{remark} $\iota' \circ \alpha_{\sigma\tau} = \alpha_{\tau} \circ \beta$ where $\iota'$ is
the inclusion $\iota': \ds{V_{\sigma\tau}} \rightarrow \ds{V_{\tau}}$.  Indeed, since
 \[k_{V_{\sigma\tau}}\circ \alpha_{\sigma\tau} = r \circ s_{\sigma\tau} = r \circ s_{\tau}\circ \beta\ = k_{V_\tau}\circ \alpha_{\tau}\circ \beta\]
and $k_{V_{\sigma\tau}} = k_{V_{\tau}}\circ \iota'$, we get our result because $k_{V_{\sigma\tau}}$ is mono.

\end{remark}

In terms of its inverse image, $\beta$ has the following simple alternative definitions:

\begin{lemma}
$\beta^*(a) = s_\sigma^*(r^*(a))=s_\sigma^*(a^*)=\alpha_\sigma^*(a\wedge V_\sigma)$ for all $a\in\ds{U_\tau}$.
\end{lemma}

\begin{proof}
The diagram that defines $\beta$ yields the following condition in terms of frame homomorphisms, for all $a\in\rs(Q)$:
\[U_{\sigma\tau}\wedge (r\circ s_\sigma)^*(a) = \beta^*(a\wedge U_\tau)\;.\]
Hence, for $a\in\ds{U_\tau}$ we obtain
\[\beta^*(a) = \beta^*(a\wedge U_\tau) = U_{\sigma\tau}\wedge (r\circ s_\sigma)^*(a) = (r\circ s_\sigma)^*(U_{\tau})
\wedge (r\circ s_\sigma)^*(a) = (r\circ s_\sigma)^*(a)\;.\]
The rest follows from the general relations involving $\sigma^*$ and $\alpha^*$. \qed
\end{proof}

We have the following straightforward ``$t$-version'' of the definition \ref{def:product} of product of local bisections:

\begin{lemma}\label{prop:inverse1}
Let $\sigma$ and $\tau$ be local bisections. Then
\[ t_{\sigma\tau} = m \circ \langle t_{\sigma}\circ \beta', t_{\tau}\circ \iota'\rangle\;,\]
where $\iota': \ds{V_{\sigma\tau}} \rightarrow \ds{V_{\tau}}$ is the restriction and
$\beta':\ds{V_{\sigma\tau}} \rightarrow \ds{V_{\sigma}}$ is the pullback, in the category of locales,
of $d \circ t_{\tau}$ along $k_{V_{\sigma}}$ as in the diagram:
\[
\xymatrix{
\ds{V_{\sigma\tau}}\ar[rr]^{\beta'}\ar@{_(->}[d]_{\iota'}
&&\ds{V_\sigma}\ar@{_(->}[d]_{k_{V_{\sigma}}}\\
\ds{V_\tau}\ar[rr]_{d\circ t_{\tau}}&&\rs(Q)\;.
}
\]
\end{lemma}

The following propositions state useful technical properties of the product of local bisections.

\begin{lemma}\label{prop:inverse2}
$(\sigma\tau)^{-1} = \tau^{-1}\sigma^{-1}$ for all local bisections $\sigma$ and $\tau$.
\end{lemma}

\begin{proof}
It is easy to see that $U_{(\sigma\tau)^{-1}} = U_{\tau^{-1}\sigma^{-1}}$:
\begin{eqnarray*}
U_{(\sigma\tau)^{-1}} &=& V_{\sigma\tau} = t^*_{\tau}(d^*(V_{\sigma})) = t^*_{\tau}(V_{\sigma})\\
&=& s^*_{\tau^{-1}}(V^*_{\sigma}) = s^*_{\tau^{-1}}(U^*_{\sigma^{-1}}) = U_{\tau^{-1}\sigma^{-1}}\;.
\end{eqnarray*}
Also, we have
\begin{eqnarray*}
s_{(\tau^{-1}\sigma^{-1})^{-1}} &=& i \circ t_{\tau^{-1}\sigma^{-1}}\\
&=& i \circ m \circ \langle t_{\tau^{-1}}\circ \beta', t_{\sigma^{-1}}\circ \iota'\rangle\\
&=& m \circ \langle i \circ t_{\sigma^{-1}}\circ \iota', i \circ  t_{\tau^{-1}}\circ \beta' \rangle\;.
\end{eqnarray*}
And, using \cite[Prop.\ 2.3]{aim}, the latter equals
$m \circ \langle  s_{\sigma}\circ \iota', s_{\tau}\circ \beta' \rangle = s_{\sigma\tau}$
because in this case we have obviously $\beta' = \beta$ and $\iota' = \iota$. \qed
\end{proof}

\begin{lemma}
Let $\sigma$ and $\tau$ be local bisections. For all $a\in \ds{U_\tau}$ we have $s^*_{\sigma\tau}(a^*)=s^*_{\sigma}(s^*_\tau(a^*)^*)$.
\end{lemma}

\begin{proof}
This follows immediately from one of the above formulas for $\beta^*$ 
and from the proof of \ref{prop:localbisec}, where we have seen that
$r\circ s_{\sigma\tau} = r\circ s_\tau\circ\beta$. \qed
\end{proof}

We shall not prove further properties of $\sections(Q)$ now, since these will follow from the results at the end of this section in the more general setting of arbitrary open quantal frames. For now we shall just obtain formulas for $s_{\sigma\tau}$ in terms of its inverse image, which will be needed later:

\begin{lemma}\label{prop:product}
For all $a\in Q$:
\begin{eqnarray}
s_{\sigma\tau}^*(a) &=& s_\sigma^*\left(\V_{xy\le a}\ x\wedge r^*(s^*_\tau(y))\right) \label{formula1}\\
&=& s_\sigma^*\left(\V_{xy\le a}\ x\wedge s^*_\tau(y)^*\right)\label{formula2}\\
&=& \V_{xy\le a}\ s_\sigma^*(x)\wedge \alpha_\sigma^*(s^*_\tau(y)\wedge V_\sigma)\label{formula3}
\end{eqnarray}
\end{lemma}

\begin{proof}
The inverse image of $s_{\sigma\tau}$ is
$s^*_{\sigma\tau} = [\iota^*\circ s_\sigma^*,\beta^*\circ s_\tau^*Ê]Ê\circ m^*$.
For each $x\otimes y\in Q\otimes_{\rs(Q)} Q$ the copairing acts as follows:
\[[\iota^*\circ s_\sigma^*,\beta^*\circ s_\tau^*Ê](x\otimes y) = (\iota^*\circ s_\sigma^*)(x)\wedge (\beta^*\circ s_\tau^*)(y)\] and from the previous propositions it is straightforward to see that this coincides with the following three expressions:
\[s_\sigma^*(x\wedge r^*(s^*_\tau(y))) = s_\sigma^*(x\wedge s^*_\tau(y)^*)
= s_\sigma^*(x)\wedge\alpha_\sigma^*(s^*_\tau(y)\wedge V_\sigma)\;.\]
The rest follows from the formula for $m^*$ as the right adjoint of $\mu_0$:
\[m^*(a) = \V_{xy\le a} x\otimes y\;. \qed\]
\end{proof}

\paragraph{Local bisections of inverse quantal frames.}\label{section:Local Bisections and Partial Units}

Now let $Q$ be an inverse quantal frame. As already mentioned in section \ref{sec:groupoidquantales}, the partial units of $Q$ correspond bijectively with the local bisections of its associated \'etale groupoid $\groupoid(Q)$, hence with the local bisections of $Q$.
Let us make this correspondence, which is well known and obvious for topological \'etale groupoids, explicit in the case of localic \'etale groupoids, in particular showing that, as expected, there is an isomorphism of inverse semigroups $\ipi(Q)\cong\sections(Q)$.

\begin{lemma}\label{monoids}
There is a homomorphism of involutive monoids
\[\xi:\sections(Q) \to Q\]
defined by $(s,U)\mapsto s_!(U)$.
\end{lemma}

\begin{proof}
The local bisections of $Q$ are the local bisections of the \'etale groupoid $\groupoid(Q)$, which are local sections of the local homeomorphism $d$ and thus are open. Hence, there is a map
\[\xi:\sections(Q)\to Q\]
defined by $\xi(s,U)=s_!(U)$.
Let us prove that $\xi$ is a homomorphism of involutive monoids. Using the definition of $\beta$ we obtain:
\begin{eqnarray*}
\xi(\sigma\tau) &=& s_{\sigma\tau!}(U_{\sigma\tau})) = m_!\langle s_{\sigma!}\iota_{!}(U_{\sigma\tau}),s_{\tau!} \beta_!(U_{\sigma\tau})\rangle\\
&=& m_!\langle s_{\sigma!}(U_{\sigma\tau}),s_{\tau!}\beta_!(U_{\sigma\tau})\rangle\;.
\end{eqnarray*}
But, since
$U_{\sigma\tau} = {s_\sigma}^*(r^*(U_\tau))$,
we have
\[
s_{\sigma!}(U_{\sigma\tau}) =  s_{\sigma!}( {s_\sigma}^*(r^*(U_\tau))) \leq r^*(U_\tau) = r^*(d_! s_{\tau!}(U_\tau))
\]
and
\begin{eqnarray*}
s_{\tau!}\beta_!(U_{\sigma\tau})& \leq& d^* \circ d_!  s_{\tau!}\beta_!(U_{\sigma\tau}) = d^* \circ r_! \circ s_{\sigma!} \iota_!(U_{\sigma\tau})\\
&=& d^* \circ r_! \circ s_{\sigma!}(U_\sigma)\;.
\end{eqnarray*}
Furthermore,
\begin{eqnarray*} s_{\sigma!}(U_{\sigma\tau}) &=& s_{\sigma!}(U_{\sigma}) \wedge  r^*(d_! s_{\tau!}(U_\tau))\\
&=& s_{\sigma!}(U_{\sigma}) \wedge 1( s_{\tau!}(U_\tau))^*\;,
\end{eqnarray*}
because since $s_\sigma$ is open we can use the Frobenius condition for
$s_\sigma$ to obtain
\[s_{\sigma!}(U_{\sigma})\wedge r^*(U_{\tau}) = s_{\sigma!}(U_{\sigma}\wedge 
s^*_{\sigma}( r^*(U_{\tau}))) = s_{\sigma!}(s^*_{\sigma}( 
r^*(U_{\tau}))) =  s_{\sigma!}(U_{\sigma\tau})\]
and
\begin{eqnarray*} s_{\tau!}\beta_!(U_{\sigma\tau}) &=& s_{\tau!}(U_\tau) \wedge d^* \circ r_! \circ s_{\sigma!}(U_\sigma)\\
&=& s_{\tau!}(U_\tau) \wedge  (s_{\sigma!}(U_\sigma))^{*}1\;,
\end{eqnarray*}
using analogously $\beta_!(U_{\sigma\tau}) = s^*_{\tau}(d^*(V_{\sigma}))
= s^*_{\sigma}d^*(r_{!}s_{\sigma!}(U_{\sigma}))$.
Hence, using (B) we get
\begin{eqnarray*}
\xi(\sigma\tau) &=&  (s_{\sigma!}(U_{\sigma}) \wedge 1( s_{\tau!}(U_\tau))^*)(s_{\tau!}(U_\tau) \wedge  (s_{\sigma!}(U_\sigma))^{*}1)\\
&=& (s_{\sigma!}(U_{\sigma}) \wedge 1( s_{\tau!}(U_\tau))^* \wedge 1(s_{\sigma!}(U_\sigma)))s_{\tau!}(U_\tau)\;,
\end{eqnarray*}
and, using (B) again, this equals
\begin{eqnarray*}
(s_{\sigma!}(U_{\sigma}) \wedge 1( s_{\tau!}(U_\tau))^*)s_{\tau!}(U_\tau)&=& s_{\sigma!}(U_{\sigma})(s_{\tau!}(U_\tau)1 \wedge s_{\tau!}(U_\tau))\\
&=& s_{\sigma!}(U_{\sigma}) s_{\tau!}(U_\tau)\;.
\end{eqnarray*}
Finally, $\xi$ is involutive because
\begin{eqnarray*}
\xi(\sigma^{-1}) &=& s_{\sigma^{-1}!}(V) = i_! \circ t_{\sigma!}(V)  = i_! \circ s_{\sigma !} \circ {\alpha_!}^{-1}(V) \\
&=& s_{\sigma!}(U)^* = \xi(\sigma)^*\;. \qed
\end{eqnarray*}
\end{proof}

\begin{theorem}
The homomorphism $\xi:\sections(Q)\to Q$ restricts to an isomorphism of involutive monoids $\sections(Q)\to\ipi(Q)$.
\end{theorem}

\begin{proof}
Letting $a$ be a partial unit of $Q$, we obtain a local bisection as follows. First we remark that the restriction of the support to $\downsegment a$ defines a frame isomorphism $\spp_a:\downsegment a\to\downsegment\spp(a)$, for if $b\le\spp(a)$ we have $\spp(ba)=\spp(b\spp(a))=b$, and if $x\le a$ we have
$\spp(x)a=x$, hence showing that the monotone map $\downsegment\spp(a)\to\downsegment a$ given by $b\mapsto ba$ is the inverse of $\spp_a$. Then we obtain a map of locales $s:\tilde U\to Q$ by defining $U= \spp(a)1$ and $s^*(x)=\spp(x\wedge a)1$. (That is, $s^*$ is the composition
\[\xymatrix{
Q\ar[rr]^-{(-)\wedge a}&&\downsegment a\ar[rr]^-{\cong}&&\downsegment\spp(a)
\ar[rr]^-{(-)1}&&\tilde U
}\]
of frame homomorphisms.) It is straightforward to verify that the pair $(s,U)$ thus obtained is a local bisection, and thus we have obtained a map
\[\zeta:\ipi(Q)\to\sections(Q)\;.\]
Now it is also straightforward to verify that the image of $\xi$ lies in $\ipi(Q)$ and that $\zeta$ is inverse to $\xi$. \qed
\end{proof}

\paragraph{Back to general open quantal frames.}

Now let us drop multiplicativity and again assume that $Q$ is just an arbitrary open quantal frame. We no longer have the maps $u$ or $m$ but we can still write the expressions (\ref{formula1})--(\ref{formula3}) for $s^*_{\sigma\tau}$:

\begin{lemma}\label{prop:map}
Let $\sigma$ and $\tau$ be local bisections of $Q$, and define the function
\[f:Q\to Q\]
by
\begin{equation}f(a) = \V_{xy\le a} s^*_\sigma(x\wedge s^*_\tau(y)^*)\;. \label{formula4}\end{equation}
\begin{enumerate}
\item For all $a\in\rs(Q)$ we have
\begin{equation}f(a^*)\ge s^*_\sigma(s^*_\tau(a^*)^*)\;.\label{formula5}\end{equation}
\item $f$ preserves binary meets.
\item If $f$ preserves joins then it is the inverse image homomorphism of a locale map
$w:W \to Q$
where $W$ is given by the formula 
\[W= s^*_\sigma(r^*(U_\tau))=s_\sigma^*(U_\tau^*)\;.\]
In this case we can define a local bisection $\sigma\tau$ by $s^{*}_{\sigma\tau} = f$ and $U_{\sigma\tau} = W$.
\end{enumerate}
\end{lemma}

\begin{proof}
If $a\in \rs(Q)$ we have, since $Q$ is open, $a=a1$, whence $a^*=1a^*$. Letting $x=1$ and $y=a^*$ we have $xy\le a$ and therefore we obtain
\[f(a^*)\ge s^*_\sigma(1\wedge s^*_\tau(a^*)^*)=s^*_\sigma(s^*_\tau(a^*)^*)\;.\]
In order to see that $f$ preserves binary meets we remark that $f(a\wedge b)\le f(a)\wedge f(b)$ because $f$ is monotone, and, for the converse inequality,
\begin{eqnarray*}
f(a)\wedge f(b) &=& \V_{\scriptsize\begin{array}{c}xy\le a\\ zw\le b\end{array}} s^*_\sigma(x\wedge s^*_\tau(y)^*)\wedge s^*_\sigma(z\wedge s^*_\tau(w)^*)\\
&=&\V_{\scriptsize\begin{array}{c}xy\le a\\ zw\le b\end{array}} s^*_\sigma(x\wedge z\wedge s^*_\tau(y\wedge w)^*)\\
&\le&\V_{\scriptsize\begin{array}{c}(x\wedge z)(y\wedge w)\le a\\ (x\wedge z)(y\wedge w)\le b\end{array}} s^*_\sigma(x\wedge z\wedge s^*_\tau(y\wedge w)^*)\\
&=&\V_{xy\le a\wedge b} s^*_\sigma(x\wedge s^*_\tau(y)^*)\\
&=&f(a\wedge b)\;.
\end{eqnarray*}
Also,
\[f(1) = \V_{xy\le 1} s^*_\sigma(x\wedge s^*_\tau(y)^*)
= s^*_\sigma(1\wedge s^*_\tau(1)^*)
=s^*_\sigma(U_\tau^*)=W\;,\]
and thus if $f$ preserves joins it defines a locale map
\[w:\ds{W}\to Q\]
by $w^*=f$. By \ref{prop:product} we have
$f = [\iota^* \circ s^*_{\sigma}, \beta^*\circ s^*_{\tau}] \circ (\mu_0)_*$.
Then using \ref{cor:framehoms} we get
\[ f \circ d^* = [\iota^* \circ s^*_{\sigma}, \beta^*\circ s^*_{\tau}] \circ (\mu_0)_* \circ d^* = \iota^* \circ s^*_{\sigma}\circ d^*
= \iota^*\circ k^*_{U_{\sigma}} = k^{*}_{U_{\sigma\tau}}\;,\]
and thus $d \circ w = k_{U_{\sigma\tau}}$. Also,
\[f \circ r^* = [\iota^* \circ s^*_{\sigma}, \beta^*\circ s^*_{\tau}] \circ (\mu_0)_* \circ r^* = \beta^*\circ s^*_{\tau}\circ r^*
= \beta^* \circ \alpha_{\tau}^*\circ k^*_{V_{\tau}}\;,\]
and thus $f \circ r^*$ is the frame homomorphism that determines the open regular monomorphism of locales $k_{V_\tau} \circ \alpha_\tau \circ \beta$.
\qed
\end{proof}

Looking at the proofs of \ref{prop:inverse1} and \ref{prop:inverse2}, and using the fact
that for a general open quantal frame we have $(\mu_0)_*(a^*) = \V_{xy \leq a} y^*\otimes x^*$, we get
the following generalization of \ref{prop:inverse2}:

\begin{lemma}\label{prop:inversion}
Let $\sigma$ and $\tau$ be local bisections such that $f$, as defined in \ref{prop:map}, preserves joins (and thus the product $\sigma\tau$ is well defined). Then the product
$\tau^{-1}\sigma^{-1}$ is well defined and we have
$(\sigma\tau)^{-1} = \tau^{-1}\sigma^{-1}$.
\end{lemma}

\section{Weak multiplicativity}\label{section4}

Now we begin to study the extent to which the local bisections of an open quantal frame $Q$ act on $Q$. Along with this we introduce a condition, called \emph{weak multiplicativity}, which implies that $\sections(Q)$ has a well defined multiplication and therefore is an inverse semigroup.

\paragraph{Actions of inverse semigroups on quantales.}

We begin by presenting the definition of an involutive action of an inverse semigroup on
an involutive quantal frame.

\begin{definition} Let $S$ be an inverse semigroup and $Q$ an involutive quantal frame. An \emph{involutive action of $S$ on $Q$}, or simply an \emph{$S$-action on $Q$}, is a map
\[\cdot: S \times Q \rightarrow Q\]
such that, if for all $a\in Q$ and $s\in S$ we write
\[ a \cdot s = (s^{-1} \cdot a^*)^*\;,\]
we have, for all $s,t\in S$ and $a,b\in Q$:
\begin{itemize}
\item $s \cdot (-)$ preserves joins;
\item $s \cdot (t \cdot a) = (st)\cdot a$;
\item $a(s \cdot b) = (a \cdot s)b$.
\end{itemize}
\end{definition}

Here are some elementary properties:

\begin{lemma}\label{propinvo}
Consider an involutive action of an inverse semigroup $S$ on an inverse quantal frame $Q$. We have, for all $s,t\in S$ and $a,b\in Q$:
\begin{itemize}
\item $(-)\cdot s$ preserves joins;
\item $(a \cdot s) \cdot t = a \cdot (st)$;
\item $ (a \cdot s)(t \cdot b) = (a \cdot (st))b = a((st) \cdot b)$.
\end{itemize}
\end{lemma}

\begin{example}\label{exm:IQaction}
Let $Q$ be an inverse quantal frame. Then it is clear that we have an $\ipi(Q)$-action
on $Q$ which is just the restriction of the multiplication of $Q$. Equivalently, there is a $\sections(Q)$-action on $Q$.
\end{example}

\begin{example}\label{exm:GQactiontopol}
Let $G$ be an open topological groupoid. A local bisection of $G$ is a pair $(s,U)$ consisting of an open set $U$ of $G_0$ and a continuous local section $s:U\to G_1$ of $d$ such that $r\circ s$ is injective (this is not the same as a local bisection of the quantale $\opens(G)$ unless $G$ is sober). Similarly to localic groupoids, we write $t:V\to G_1$ for the corresponding local section of $r$, and we denote the inverse semigroup of local bisections of $G$ by $\sections(G)$. Then an involutive action of $\sections(G)$ on the quantale $\opens(G)$ is defined, for all $\sigma=(s,U)\in\sections(G)$ and all $W\in\opens(G)$, by pointwise multiplication:
\[\sigma\cdot W = \{s(x)y\st x\in U,\ y\in W,\ r(s(x))=d(y)\}\;.\]
It is straightforward to verify that all the axioms of a $\sections(G)$-action are satisfied. In order to see that for each open set $W\subset G_1$ the set $\sigma\cdot W$ is indeed open consider the map $\lambda_\sigma:d^{-1}(V)\to G_1$ defined by
\[\lambda_\sigma(y)=t(d(y))y\;.\]
It is easy to see that this is an open map, and thus $\sigma\cdot W$, which equals $\lambda_\sigma(W)$, is open.
\end{example}

\paragraph{Actions of local bisections on open groupoids.}

Let $G$ be an arbitrary but fixed open localic groupoid, and let us write  $Q=\opens(G)$ for its multiplicative open quantal frame. Let also $\sigma$
be a local bisection of $Q$. The notation $U$, $V$, $s$, $t$, and $\alpha$, with or without subscripts, will be used as before.

Adapting the definition of the map $\lambda_\sigma$ of the example in \ref{exm:GQactiontopol} we get:

\begin{definition}
$\lambda_\sigma:\downsegment V \to G_1$ is the map of locales defined by
\[\lambda_\sigma = m\circ\langle t\circ d\vert_{V},\iota_{d^*(V)}\rangle\;,\]
where $\iota_{d^*(V)}$ is the inclusion of $d^*(V)$ into $G_1$ as an open sublocale, and $d\vert_{V}$ is the pullback of $d$ along $k_{V}$.
We call $\lambda_\sigma$ the \emph{left action} of $\sigma$ on $G$.
\end{definition}

\begin{definition}
Likewise we define $\rho_\sigma:\downsegment U^*\to G_1$ given by
\[\rho_\sigma = m\circ\langle \iota_{r^*(U)},s\circ r\vert_{U}\rangle\;,\]
which we shall call the \emph{right action} of $\sigma$ on $G$, where $ \iota_{r^*(U)}$ and $r\vert_{U}$ have
the obvious meaning.
\end{definition}

Analogously to \ref{exm:GQactiontopol}, it can be shown that:
\begin{lemma}
$\lambda_\sigma$ and $\rho_\sigma$ are open regular monomorphisms whose images are, respectively, the open sublocales $\downsegment U$ and $\downsegment V^*$.
\end{lemma}

\begin{definition} For all $a \in \opens(G)$ we define
\begin{eqnarray*}
\sigma\cdot a &=& {(\lambda_\sigma)}_!(a\wedge V)\;,\\
a\cdot \sigma &=& {(\rho_{\sigma})}_!(a\wedge U^*)\;.
\end{eqnarray*}
\end{definition}

We remark that, contrary to \ref{exm:GQactiontopol}, this does not necessarily define an involutive action of $\sections(G)$ on $\opens(G)$, and, indeed, much of what we shall do later in section \ref{sec5} has to do with conditions under which such an involutive action exists. For now let us record a useful property:

\begin{lemma}\label{futuredef}
For all $a\in\opens(G)$ we have
\begin{eqnarray*}
\sigma\cdot a &=& \V_{x^*y\le a}s^*(x)\wedge y\;,\\
a \cdot \sigma^{-1} &=& \V_{xy\le a} x \wedge s^*(y)^*\;.
\end{eqnarray*}
\end{lemma}

\begin{proof}
The inverse images of $\lambda_\sigma$ and $\rho_\sigma$ are:
\begin{eqnarray*}
\lambda_\sigma^*(a) &=& \V_{xy\le a} t^*(x)\wedge y\;,\\
\rho_\sigma^*(a) &=& \V_{xy\le a} x \wedge s^*(y)^*\;.
\end{eqnarray*}
But we also have
\begin{eqnarray*}
\sigma\cdot a &=& {\lambda_{\sigma^{-1}}}^*(a\wedge V)\;,\\
a\cdot \sigma &=&  {\rho_{\sigma^{-1}}}^*(a\wedge U^*)\;,
\end{eqnarray*}
and thus using the formula $t_{\sigma^{-1}}=i\circ s_\sigma$ we ultimately obtain
\begin{eqnarray*}
\sigma\cdot a &=& \V_{x^*y\le a}s^*(x)\wedge y\;,\\
a \cdot \sigma^{-1} &=& \V_{xy\le a} x \wedge s^*(y)^*\;. \qed
\end{eqnarray*}
\end{proof}

\paragraph{Actions of local bisections on open quantal frames.}

Now we drop multiplicativity and consider $Q$ to be just an open quantal frame.
Inspired by \ref{futuredef} we are led to the following definition.

\begin{definition}\label{definition:action}
Let $\sigma$ be a local bisection of $Q$. We define
\begin{eqnarray*}
\sigma\cdot a &=& \V_{x^*y\le a}s^*(x)\wedge y\;,\\
a \cdot \sigma^{-1} &=& \V_{xy\le a} x \wedge s^*(y)^*\;.
\end{eqnarray*}
\end{definition}

We immediately obtain the following two technical conditions:

\begin{lemma}\label{lemma:star}  For all $\sigma,\tau\in\sections(Q)$ we have
$(\sigma \cdot a)^* = a^* \cdot \sigma^{-1}$.
\end{lemma}

\begin{proof}
\begin{eqnarray*}
(\sigma \cdot a)^* &=&  \left(\V_{x^*y\le a}s^*(x)\wedge y\right)^*  =  \V_{x^*y\le a}s^*(x)^*\wedge y^* =  \V_{y^*x\le a^*}s^*(x)^*\wedge y^*\\
&=& a^* \cdot \sigma^{-1}\;. \qed
\end{eqnarray*}
\end{proof}

\begin{lemma}\label{prop:map2}
The map $f$ of \ref{prop:map} satisfies
\[f(a) = s^{*}_{\sigma}(a \cdot \tau^{-1})\;.\]
In particular, if $f$ preserves joins $\sigma\tau$ is defined and we have
\[s^*_{\sigma\tau}(a) = s^*_\sigma(a\cdot\tau^{-1})\;.\]
\end{lemma}

\begin{proof}
This follows from
\[f(a)=s^*_\sigma\left(\V_{xy\le a} x\wedge {s^*_\tau(y)}^*\right) = s^*_\sigma(a\cdot\tau^{-1})\;.\qed\]
\end{proof}

\paragraph{Weak multiplicativity.}

Continuing to consider $Q$ to be an arbitrary open quantal frame, not necessarily multiplicative, we now introduce a weak form of multiplicativity under which $\sections(Q)$ will be seen to have a well defined inverse semigroup structure.

\begin{definition} The open quantal frame $Q$ is called
\emph{weakly multiplicative} if the following conditions hold for all $\sigma,\tau,\nu \in \sections(Q)$:
\begin{enumerate}
\item $\upsilon$ preserves joins;
\item $\sigma \cdot(-)$ preserves joins;
\item $(\sigma\tau)\nu = \sigma(\tau\nu)$.
\end{enumerate}
\end{definition}

We remark that the notion of weak multiplicativity includes, by definition, the condition that the map $u$ exists, with inverse image $u^*=\upsilon$, along with the existence of the global bisection $\varepsilon=(u,1)$ (\cf\ \ref{globalbisec}). This condition alone has several consequences regarding the existence of well defined products of certain local bisections, as the following three lemmas illustrate.

\begin{lemma}
Let $Q$ be weakly multiplicative. For all $\sigma\in\sections(Q)$ we have the following well defined products:
\begin{eqnarray*}
\sigma\sigma^{-1} &=& (U_\sigma, u\circ k_{U_\sigma})\;,\\
\sigma^{-1}\sigma&=&(V_\sigma,u\circ k_{V_\sigma})\;.
\end{eqnarray*}
\end{lemma}

\begin{proof}
First we show that the domain of $\sigma\sigma^{-1}$ is what it should be:
\[U_{\sigma\sigma^{-1}} = s^*(U_{\sigma^{-1}}^*)=s^*(V_\sigma^*) = U_\sigma\;.\]
Now the domain section of $\sigma\sigma^{-1}$ (\cf\ formula for $f$ in \ref{prop:map}):
\begin{eqnarray*}
s^*_{\sigma\sigma^{-1}}(a) &=& \V_{xy\le a} s^*_\sigma(x)\wedge s^*_\sigma(s^*_{\sigma^{-1}}(y)^*)\\
&=&\V_{xy\le a} s^*_\sigma(x)\wedge s^*_\sigma(t^*_{\sigma}(y^*)^*)\\
&=&\V_{xy\le a} s^*_\sigma(x)\wedge \alpha^*_\sigma(t^*_{\sigma}(y^*))\\
&=& \V_{xy\le a} s^*_\sigma(x)\wedge s^*_\sigma(y^*)\\
&=&s^*_\sigma \left(\V_{xy^*\le a} x\wedge y\right)\\
&=& s^*_\sigma(d^*(u^*(a)))\\
&=& (u\circ k_{U_\sigma})^*(a)\;.
\end{eqnarray*}
Hence, $s_{\sigma\sigma^{-1}}=u\circ k_{U_\sigma}$.
For $\sigma^{-1}\sigma$ everything is analogous. \qed
\end{proof}

\begin{lemma}\label{prop:rightunit}
Let $Q$ be weakly multiplicative and let $\sigma$ be a local bisection. Then the product $\sigma\varepsilon$ is well defined and we have $\sigma\varepsilon = \sigma$.
\end{lemma}

\begin{proof}
For the domain we have $U_{\sigma\varepsilon}=s^*_\sigma(U_\varepsilon^*)=s^*_\sigma(1)=U_\sigma$; and, for the domain section (\cf\ \ref{prop:map}):
\[
s^*_{\sigma\varepsilon}(a) =s_\sigma^*\left(\V_{xy\le a} x\wedge u^*(y)^*\right)\;.
\]
In order to show that $s_{\sigma\varepsilon}=s_\sigma$ we shall prove that the argument of $s^*_\sigma$ in the last expression equals $a$:
\begin{eqnarray*}
\V_{xy\le a} x\wedge u^*(y)^* &=& \V_{xy\le a} x\wedge \left(\V_{zz^*\le y} z\right)^*\\
&=&\V_{xzz^*\le a} x\wedge z^*\\
&=&\V_{xz^*z\le a} x\wedge z\\
&=&\V_{xx^*x\le a} x\\
& =& a\;.
\end{eqnarray*}
The last step is the axiom (U). \qed
\end{proof}

In a similar way we prove the following:

\begin{lemma}\label{prop:associative}
Let $Q$ be weakly multiplicative and let $\sigma$ be a local bisection of $Q$. Then the product $\sigma(\sigma^{-1}\sigma)$ is well defined and we have $\sigma(\sigma^{-1}\sigma)=\sigma$.
\end{lemma}

\begin{proof}
For the domain we have $U_{\sigma(\sigma^{-1}\sigma)}=s^*_\sigma(V_\sigma^*)=U_\sigma$; so let us check the domain section (\cf\ \ref{prop:map}):
\begin{eqnarray*}
s^*_{\sigma(\sigma^{-1}\sigma)}(a) &=&
s_\sigma^*\left(\V_{xy\le a} x\wedge s^*_{\sigma^{-1}\sigma}(y)^*\right)\\
&=&s_\sigma^*\left(\V_{xy\le a} x\wedge V^*_\sigma\wedge u^*(y)^*\right)\\
&=&s_\sigma^*\left(\V_{xy\le a} x\wedge u^*(y)^*\right)\\
&=&s_\sigma^*(a)\;.
\end{eqnarray*}
The last step follows from (U). \qed
\end{proof}

Now we arrive at the main results of this section.

\begin{theorem}\label{lemma:properties}\label{lemma:pseudogroup}
If $Q$ is weakly multiplicative $\sections(Q)$ 
is a complete and infinitely distributive inverse semigroup (\ie, an abstract complete pseudogroup) and the following conditions are satisfied.
\begin{enumerate}
\item\label{cond1} We have $\epsilon\cdot a = a \cdot \epsilon = a$ for all $a\in Q$.
\item The natural order of $\sections(Q)$ is given by restriction.
\item\label{cond3} The semilattice of idempotents of $\sections(Q)$ is isomorphic to $\rs(Q)$.
\item\label{cond4} If $\sigma \leq \tau$ then $\sigma \cdot a \leq \tau \cdot a$ for all $a \in Q$.
\end{enumerate}
\end{theorem}

\begin{proof}

Let us assume that $Q$ is weakly multiplicative. We begin by observing that, by
\ref{prop:map2}, the multiplication is well defined, and it is associative by hypothesis.
That we have involutivity follows from \ref{prop:inversion}.
Since we have the equality of \ref{prop:associative},
in order to obtain an inverse semigroup we need only show that the idempotents of the form $\sigma\sigma^{-1}$ commute. 
We have
\[s_{(\sigma\sigma^{-1})(\tau\tau^{-1})}^*(a)  = \upsilon\left(\V_{xy\leq a} x \wedge \upsilon(y)^*\wedge U_{\tau}^*\right)\wedge U_{\sigma}\;,\]
since $\sigma\sigma^{-1} = (u \circ k_{U_{\sigma}},U_{\sigma})$ and  $\tau\tau^{-1} = (u \circ k_{U_{\tau}},U_{\tau})$.
It is then easy to see that
\begin{eqnarray*}
\upsilon\left(\V_{xy\leq a} x \wedge \upsilon(y)^*\wedge U_{\tau}^*\right)\wedge U_{\sigma} &=&  \upsilon\left(\V_{xy\leq a} x \wedge \upsilon(y)^*\wedge 
U_{\sigma}^*\right)\wedge U_{\tau}\\
& = & s_{(\tau\tau^{-1})(\sigma\sigma^{-1})}^*(a)\;,
\end{eqnarray*}
using the fact that $\upsilon(z^*) = \upsilon(z) = z$ for $z \in \rs(Q)$. 
It is obvious that $U_{(\sigma\sigma^{-1})(\tau\tau^{-1})} = U_{(\tau\tau^{-1})(\sigma\sigma^{-1})}$.

We have thus concluded that $\sections(Q)$ is an inverse semigroup. Let us prove conditions \ref{cond1}--\ref{cond4}. The condition $\epsilon \cdot a = a$ follows directly from (U).
Suppose $\sigma \leq \tau $. If $\sigma = (\rho\rho^{-1})\tau$ then
\begin{eqnarray*}
s_{\sigma}^*(a) &=& u^*\left(\V_{xy \leq a} x \wedge s_{\tau}^*(y)^*\right)\wedge U_{\rho}\\
& =& \left(\V_{xy \leq a} u^*(x) \wedge u^*s_{\tau}^*(y)^*\right)\wedge U_{\rho}\\
& =& \left(\V_{xy \leq a} s_{\tau}^*u^*(x) \wedge s_{\tau}^*(y)^*\right)\wedge U_{\rho}\\
& =&s_{\tau}^*\left(\V_{xy \leq a} u^*(x) \wedge y\right)\wedge U_{\rho}\\
& =& s_{\tau}^*(a) \wedge  U_{\rho} = (s_{\tau}\circ k_{U_{\rho}})^*(a)
\end{eqnarray*}
and $U_{\sigma} \leq U_{\tau}$.
Now notice that all the idempotents are necessarily of the form $\sigma\sigma^{-1}$.
In order to show that the naturally ordered set of idempotents of the form $\sigma\sigma^{-1}$ is order isomorphic to $\rs(Q)$
we use the monotone assignments
\[\sigma\sigma^{-1} \mapsto U\]
and
\[U \mapsto u \circ k_U\;. \]
For condition \ref{cond4} suppose $\sigma \leq \tau$. Then, as we have seen,
\[s_{\sigma}^*(w) =  s_{\tau}^*(w) \wedge  U_{\rho}\;.\]
Hence,
\begin{eqnarray*}
\sigma\cdot a &=& \V_{x^*y\le a}s_{\sigma}^*(x)\wedge y = \V_{x^*y\le a}s_{\tau}^*(x)\wedge y\wedge  U_{\rho}\\
&\leq&  \V_{x^*y\le a}s_{\tau}^*(x)\wedge y = \tau \cdot a\;.
\end{eqnarray*}

Finally, $\sections(Q)$ is infinitely distributive due to condition \ref{cond3}, which implies that $E(\sections(Q))$ is a frame, and completeness follows from a standard argument of gluing of local sections applied to bisections:
let $(\sigma_i)_{i\in J}$ be a family of compatible elements $(s_i,U_i)\in\sections(Q)$, that is, such that for all $i,j\in I$ both $\sigma_i\sigma_j^{-1}$ and $\sigma_i^{-1}\sigma_j$ are idempotents;
it follows from the previous results that for any $i,j \in I$ we have 
\begin{eqnarray*}
\sigma_i\vert_{U_i \wedge U_j} &=&\sigma_j\vert_{U_i \wedge U_j}\\
\sigma_i^{-1}\vert_{V_i \wedge V_j} &=& \sigma_j^{-1}\vert_{V_i \wedge V_j}\;,
\end{eqnarray*}
and thus
there is a gluing $\sigma=(s,U)$ of the family $(\sigma_i)$, which is the join $\V_i\sigma_i$ in the natural order of $\sections(Q)$ and is defined
by $s^* = \V_i s^*_i$ and $U=\V_i U_i$ (\cf\ \cite[pp.\ 90--92]{Borceux}). \qed
\end{proof}

We remark that it is unknown whether all open quantal frames
are weakly multiplicative or not. But at least, as the terminology suggests, multiplicativity implies weak multiplicativity:

\begin{theorem}
If $Q$ is multiplicative then it is weakly multiplicative.
\end{theorem}
 
\begin{proof}
The first condition follows from the fact that $\sigma \cdot a = \lambda^*_{\sigma^{-1}}(a \wedge V)$.
We now check the associativity.
Consider the following frame homomorphisms, whose definitions are analogous to those of $\iota$ and $\beta$ in \ref{def:product}:
\begin{eqnarray*}
\iota_1&:& U_{(\sigma\tau)\nu} \rightarrow U_{\sigma\tau}\\
\beta_1&:& U_{(\sigma\tau)\nu} \rightarrow U_{\nu}\\
\iota_2&:& U_{\sigma\tau} \rightarrow U_{\sigma}\\
\beta_2&:& U_{\sigma\tau} \rightarrow U_{\tau}\\
\iota_3&:& U_{\sigma(\tau\nu)} \rightarrow U_{\sigma}\\
\beta_3&:& U_{\sigma(\tau\nu)} \rightarrow U_{\tau\nu}\\
\iota_4&:& U_{\tau\nu} \rightarrow U_{\tau}\\
\beta_4&:& U_{\tau\nu} \rightarrow U_{\nu}\;.
\end{eqnarray*}
First we show that
\[ \beta_2 \circ \iota_1 = \iota_4 \circ \beta_3\;, \]
by proving that
\[ k_{U_{\tau}} \circ \beta_2 \circ \iota_1 = k_{U_{\tau}}\circ \iota_4 \circ \beta_3\;, \]
which in turn follows from the following derivation:
\begin{eqnarray*}
k_{U_{\tau}} \circ \beta_2 \circ \iota_1 &=& r \circ s_{\sigma} \circ \iota_2 \circ \iota_1 = r \circ s_{\sigma} \circ \iota_3\\
&=& k_{U_{\tau\nu}}\circ \beta_3 = k_{U_{\tau}}\circ \iota_4 \circ \beta_3\;.
\end{eqnarray*}
Now we prove the associativity. First, we have
\begin{eqnarray*}
s_{(\sigma\tau)\nu} &=& m \circ \langle s_{\sigma\tau}\circ \iota_1, s_{\nu} \circ \beta_1 \rangle\\
& =& m \circ \langle m \circ \langle s_{\sigma} \circ \iota_2, s_{\tau}\circ \beta_2 \rangle \circ \iota_1, s_{\nu} \circ \beta_1 \rangle\\
& =& m \circ \langle m \circ \langle s_{\sigma} \iota_2 \circ \iota_1, s_{\tau}\beta_2 \circ \iota_1 \rangle, s_{\nu} \circ \beta_1 \rangle\;.
\end{eqnarray*}
And this equals
\begin{eqnarray*}
&&m \circ \langle m \circ \langle s_{\sigma} \iota_3, s_{\tau}\circ \iota_4 \circ \beta_3 \rangle, s_{\nu} \circ \beta_4\circ \beta_3 \rangle\\
&=& m \circ \langle  s_{\sigma} \iota_3, m \circ \langle s_{\tau}\circ \iota_4 , s_{\nu} \circ \beta_4 \rangle \circ \beta_3 \rangle\\
&=& m\circ  \langle  s_{\sigma} \iota_3, s_{\tau\nu}\circ \beta_3\rangle = s_{\sigma(\tau\nu)}\;. \qed
\end{eqnarray*}
\end{proof}

\paragraph{Weak multiplicativity revisited.}

We conclude this section by obtaining a sufficient condition for weak multiplicativity of an open quantal frame $Q$:

\begin{theorem}
Assume that  $\sigma \cdot(-)$ preserves joins for all $\sections(Q)$
 and that  the inequality $s^*_\sigma(a^*\cdot\tau^{-1}) \le  s^*_\sigma(s^*_\tau(a^*)^*)$ holds for all $a \in \rs(Q)$. Then $\sections(Q)$ has an associative multiplication. If in addition $\upsilon$ preserves joins $Q$ is weakly multiplicative.
\end{theorem}

\begin{proof}
The existence of the multiplication follows from \ref{prop:map2}.
Let us verify that $U_{(\sigma\tau)\nu}=U_{\sigma(\tau\nu)}$ for all local bisections $\sigma$, $\tau$, and $\nu$, which is easy:
\[U_{(\sigma\tau)\nu} = s^*_{\sigma\tau}(U_\nu^*) = s^*_\sigma(s^*_\tau(U_\nu^*)^*)
=s^*_\sigma(U_{\tau\nu}^*)=U_{\sigma(\tau\nu)}\;.\]
Then using the definition of the product we obtain the following, for all $a\in Q$,
\begin{eqnarray*}
s^*_{\sigma(\tau\nu)}(a) &=& \V_{xw\le a} s^*_\sigma(x\wedge s^*_{\tau\nu}(w)^*)\\
&=& \V_{xy\le a} s^*_\sigma\left(x\wedge \V_{yz\le w}s^*_\tau(y\wedge s^*_\nu(z)^*)^*\right)\\
&=&\V_{xyz\le a} s^*_\sigma(x)\wedge s^*_\sigma(s^*_\tau(y)^*)
\wedge s^*_\sigma(s^*_\tau(s^*_\nu(z)^*)^*)\;,
\end{eqnarray*}
and also
\begin{eqnarray*}
s^*_{(\sigma\tau)\nu}(a) &=& \V_{wz\le a} s^*_{\sigma\tau}(w\wedge s^*_\nu(z)^*)\\
&=&\V_{wz\le a} s^*_{\sigma\tau}(w)\wedge s^*_{\sigma\tau}(s^*_\nu(z)^*)\\
&=&\V_{wz\le a}\left( \V_{xy\le w} s^*_\sigma(x)\wedge s^*_\sigma(s^*_\tau(y)^*) \right)\wedge s^*_{\sigma\tau}(s^*_\nu(z)^*)\\
&=&\V_{xyz\le a} s^*_\sigma(x)\wedge s^*_\sigma(s^*_\tau(y)^*)
\wedge s^*_{\sigma\tau}(s^*_\nu(z)^*)\;,
\end{eqnarray*}
and thus the associativity is a consequence of the equality
\[s^*_{\sigma\tau}(s^*_\nu(z)^*) = s^*_\sigma(s^*_\tau(s^*_\nu(z)^*)^*)\]
that follows directly from \ref{prop:map} and the hypothesis. \qed
\end{proof}

\begin{corollary} If $\rs(Q)$ is a $T_1$ locale and for all $\sigma\in\sections(Q)$ both $\upsilon$ and $\sigma\cdot (-)$ preserve joins then $Q$ is weakly multiplicative.
\end{corollary}

\begin{proof}
As already seen above, we need only show that
\[s^*_\sigma(a^*\cdot\tau^{-1}) = s^*_\sigma(s^*_\tau(a^*)^*)\]
for all $a \in \rs(Q)$.
But we have, by \ref{prop:map}, that
\[s^*_\sigma(a^*\cdot\tau^{-1}) \geq s^*_\sigma(s^*_\tau(a^*)^*)\;.\]
Considering the frame homomorphisms
$f,g:\rs(Q) \rightarrow Q$ given by
$f(a) = s^*_\sigma(a^*\cdot\tau^{-1})$ (\cf\ \ref{prop:map2}) and $g(a) =  s^*_\sigma(s^*_\tau(a^*)^*)$,
we obtain $f(a) \geq g(a)$ for all $a \in \rs(Q)$ (\cf\ \ref{prop:map}). Then, since $\rs(Q)$ is $T_1$, we
have $f = g$. \qed
\end{proof}

\section{Embeddability}\label{sec5}

In this section we study sufficient conditions for a weakly multiplicative open quantal frame $Q$ to be multiplicative. These are based on embeddability properties of $Q$ into the inverse quantal frame $\lcc(\sections(Q))$ that arises as the quantale completion of $\sections(Q)$. These properties are not necessary, however, since they are not satisfied by all the multiplicative open quantal frames (in particular cases they imply localic spatiality). We also study the open groupoids $G$ whose quantales $\opens(G)$ satisfy the embeddability properties, relating this to the possibility of defining a notion of universal \'etale cover for open groupoids.

\paragraph{Ideals of inverse quantal frames.}
First we shall see how some multiplicative open quantal frames arise as ``ideals'' of inverse quantal frames. This is somewhat independent from what follows next, but the proof of the main theorem of this paper, \ref{maintheorem}, is modeled on the proof of the main theorem about ideals, \ref{ideal2}. We begin with a general definition for involutive quantal frames.

\begin{definition}
We say that a subframe $\ideal\subset Q$ of an involutive quantal frame $Q$ is an \emph{involutive ideal}
if $Q\ideal \subset \ideal$ and $\ideal^{*} \subset Q$. 
\end{definition}

An involutive ideal is in particular an involutive subquantale (not necessarily unital).

\begin{remark}
Note that since in an inverse quantal frame we have $\V \ipi(Q) = 1$ the condition $Q\ideal \subset Q$  in this case
is equivalent to
\[ \ipi(Q)\ideal \subset \ideal\;.\]
\end{remark}

\begin{remark}
When $Q$ is an inverse quantal frame
we shall want $\ideal$ to be seen as an $\rs(Q)$-module with the usual module operation
\[ z\cdot x = z \wedge x \]
for all $z \in \rs(Q)$ and $x \in \ideal$, so a natural condition to impose would be that
\[ a1 \wedge x \in \ideal\]
for all $x \in \ideal$. But this is saying precisely that
\[ \spp (a)x \in \ideal \]
for all $x \in \ideal$ and $a \in Q$. Also, $\spp( a) \in  {\downarrow}e = E(\ipi(Q)) \subset \ipi(Q)$.
So it actually follows from $\ideal$ being an involutive ideal (or $\ipi(Q)\ideal \subset \ideal$) that
$\ideal$ has the natural structure of both an $\spp( Q)$-module and an $\rs(Q)$-module.
\end{remark}

Let $Q$ be an inverse quantal frame. Henceforth we shall always denote by $\idinc :\ideal \rightarrow Q$ the inclusion monomorphism of an involutive ideal $\ideal$ into $Q$. It is obviously a homomorphism of $\rs(Q)$-modules. Now consider $\idinc \otimes \idinc = (\ident \otimes \idinc) \circ (\idinc \otimes \ident)$ given by the composition:
\begin{equation}
\vcenter{\xymatrix{
\ideal\otimes_{\rs(Q)}\ideal\ar[rr]^{\idinc \otimes \ident }&& Q
\otimes_{\rs(Q)}\ideal \ar[rr]^{\ident \otimes \idinc} && Q\otimes_{\rs(Q)}Q\;.
 }}
\end{equation}
In \cite{RR} it is shown that inverse quantal frames are projective $\spp( Q)$-modules and hence projective
$\rs(Q)$-modules (with the usual module structure). Hence, by \cite[Prop.\ II.4.1]{JT} we have that $Q$ is a flat $\rs(Q)$-module.

\begin{theorem}\label{ideal2}
Let $Q$ be an inverse quantal frame and $\ideal \subset Q$ an involutive ideal that is an
open quantal frame. Assume that $\idinc \otimes \ident$ is mono.
Then $\ideal$ is a multiplicative open quantal frame.
\end{theorem}

\begin{proof}
In order to simplify our notation we shall denote the reduced multiplication of $Q$ by $\mu$ (rather than $\mu_0$), and we shall denote the reduced multiplication of $\ideal$ by $\imu$ (this is a restriction of $\mu$). We begin by showing that
\[\mu_{*}(x) = \imu_*(x)\]
for any $x \in \ideal$. The inequality $\mu_{*}(x) \geq \imu_*(x)$ is immediate, so we need only show
that
\[\mu_{*}(x) \leq \imu_*(x)\;.\]
Now, by hypothesis,
\[\V_{\scriptsize\begin{array}{c}y\in \ideal\\ yy^{*}y \leq x\end{array}}y \geq x\]
for all $x \in \ideal$, so that
\[\mu_{*}(x) \leq \mu_{*}\left(\V_{\scriptsize\begin{array}{c}y\in \ideal\\ yy^{*}y \leq x\end{array}}y\right) = \V_{\scriptsize\begin{array}{c}y\in \ideal\\ yy^{*}y \leq x\end{array}}\mu_{*}(y)\;,\]
using the fact that $\mu_*$ preserves joins because $Q$ is inverse.
Hence we need only show that
\[\mu_{*}(y) \leq  \imu_*(x)\]
for all $y\in\ideal$ such that $yy^{*}y \leq x$ and the result will follow by taking the supremum.
Let then $y$ be such an element.
We have
\[\mu_*(y) = \V_{ab \leq y} a \otimes b\;.\]
But since $\V \ipi(Q) = 1$ we also get that $\V \ipi(Q) \otimes \ipi(Q) = 1 \otimes 1$,
so that if we show that
\[s \otimes t \leq \imu_{*}(x)\]
for all $s,t \in \ipi(Q)$ such that $st \leq y$ the result will follow by taking the supremum of all such pure tensors.
Let $s$ and $t$ be such elements. Then we have that
\[s \otimes t = ss^{*}s \otimes t = s \otimes s^{*}st \leq s\otimes s^{*}y\]
and, analogously,
\[s \otimes t \leq yt^{*} \otimes t\;.\]
Hence,
\[s \otimes t \leq (s \wedge yt^{*}) \otimes (t \wedge s^{*}y) \leq yt^{*} \otimes s^{*}y\;.\]
But by hypothesis $\ipi(Q)\ideal \subset \ideal$ so that $yt^{*} \in \ideal$ and $s^{*}y \in \ideal$.
Moreover,
\[yt^{*}s^{*}y = y(st)^{*}y \leq yy^{*}y \leq x\;.\]
Hence, combining our expressions we get that
$s \otimes t \leq \imu_{*}(x)$.

Now we show that $\imu_*$ preserves joins. 
Since $Q$ is an inverse quantal frame we have, as we have seen, that $\ident \otimes \idinc$ is mono. Hence,
by hypothesis, $\idinc \otimes \idinc = (\idinc \otimes \ident)\circ (\ident \otimes \idinc)$ is mono.
Consider a join $\V_\alpha x_\alpha$ with $x_\alpha \in \ideal$. Applying what we proved we have, since $\V_\alpha x_\alpha \in \ideal$:
\[\imu_*\left(\V_\alpha x_\alpha\right) = \mu_*\left(\V_\alpha x_\alpha\right) = \V_\alpha \mu_*(x_\alpha) = \V_\alpha \imu_{*}(x_\alpha)\;.\]
Hence, $\ideal$ is multiplicative. \qed
\end{proof}

Also, we have:

\begin{theorem}\label{ideal}
Let $\ideal$ be an involutive ideal of an inverse quantal frame $Q$ such that $\idinc \otimes \ident$ is mono and
\[\V\limits_{\scriptsize\begin{array}{c}y\in \ideal\\ yy^{*}y \leq x\end{array}}y \geq x\]
for all $x \in \ideal$.
Then $\ideal$ is a multiplicative open quantal frame.
\end{theorem}

\begin{proof}
Since $Q$ is inverse it is in particular open, so that $(B)$ and $(O)$ are verified
and hence hold also in $\ideal$. Also, in $Q$ we have
\[a \leq aa^{*}a\;,\]
so that together with the inequality in the hypothesis we get $(U)$.
Note that $\rs(\ideal) \subset \rs(Q)$ because $\ideal$ is an involutive ideal of $Q$.
Finally for $(R)$ observe that
\[\upsilon_{\ideal}  = [\ident,i] \circ \imu_* = [\ident,i] \circ {\mu_{*}}\vert_{\ideal}\;,\] 
 using the proof of the last theorem. So  the fact that $\upsilon_{Q}(q) \in \rs(Q)$
(since $Q$ is open) implies $(R)$. \qed
\end{proof}

Hence, combining the two previous results, we get our main result about ideals:

\begin{theorem}
Let $Q$ be an inverse quantal frame  and $\ideal \subset Q$ an involutive ideal. Then $\ideal$ is
a multiplicative open quantal frame if and only if the following two conditions hold:
\begin{itemize}
\item $\idinc \otimes \ident$ is mono;
\item $\V\limits_{\scriptsize\begin{array}{c}y\in \ideal\\ yy^{*}y \leq x\end{array}}y \geq x$ for all $x \in \ideal$.
\end{itemize}
\end{theorem}

It is easy to see, using the definition of $T_1$ locale together with \ref{ideal}, that

\begin{corollary}
Let $Q$ be an inverse quantal frame and $\ideal \subset Q$ a $T_1$ involutive ideal such that
$\idinc \otimes \ident$ is mono.
Then $\ideal$ is open and multiplicative.
\end{corollary}

\paragraph{Weak embeddability.}

From now on $Q$ will denote an arbitrary but fixed weakly multiplicative open quantal frame, and we shall address two embeddability properties which, jointly, are sufficient conditions for multiplicativity. Throughout the rest of this paper we shall denote by $\cover Q$ the inverse quantal frame $\lcc(\sections(Q))$ that arises as the quantale completion of the abstract complete pseudogroup $\sections(Q)$ in the sense of \cite{aim}.

\begin{lemma}\label{lemmajay}
For all $a\in Q$ the set
\[j(a)=\{\sigma \in \sections(Q)\mid s^{*}(a) = U\}\]
is a downwards closed subset of $\sections(Q)$, and it is also closed under the formation of joins of compatible subsets.
The mapping
\[j:Q\to\cover Q\]
thus defined is a homomorphism of frames.
\end{lemma}
[These properties are easy to understand in the case of the quantale of a topological open groupoid $G$: for an open set $W\subset G_1$ the set $j(W)$ is the set of local bisections $(s,U)$ whose image $s(U)$ is contained in $W$.]

\begin{proof}
Recall that $\sections(Q)$ has its  natural order given by $(s,U) \leq (s',U')$ iff $s'\vert_{U} = s$.
The sets $j(a)$ are downwards closed, for if
$\sigma \in j(a)$ and $\tau \leq \sigma$ then, since
$s_{\sigma}\vert_{U_{\tau}} = s_\tau$, if $s^{*}_{\sigma}(a) = U_{\sigma}$ then $s^*_{\tau}(a) = s^{*}_{\sigma}(a) \wedge U_{\tau} = U_{\sigma} \wedge U_{\tau} 
= U_{\tau}$.
The sets $j(a)$ are also closed for joins of families of compatible elements $(\sigma_i)$
due to the properties of the gluing $(s,U)=\V_i\sigma_i$, since we have $s^* = \V_i s^*_i$ and $U=\V_i U_i$ (\cf\ \ref{lemma:pseudogroup}).
In order to see that $j$ is a homomorphism of frames we remark that we have
\begin{eqnarray*}
j(a \wedge b) &=& j(a) \cap j(b)=j(a)\wedge j(b)\;,\\
j(1) &=& \sections(Q)=1_{\cover Q}\;,
\end{eqnarray*}
and thus $j$ preserves finite meets.
Hence, $j$ is monotone and we have
\[\V_i j(a_i) \subset j\left(\V_i a_i\right)\;.\]
The condition $\sigma \in  j\left(\V_i a_i\right)$ implies that $\V_i s^{*}(a_i)  = U$ and thus $s^{*}(a_i) \leq U$.
Hence, if we define $\sigma_i$ to be the restriction of $\sigma$ to $s^*(a_i)$ we obtain $\sigma_i \in j(a_i)$ and
$\sigma = \V_i \sigma_i \in \V_i j(a_i)$. Hence, $j$ is a homomorphism of frames.
\qed
\end{proof}

The completion $\sections(Q) \rightarrow \cover Q$ is defined by $\sigma\mapsto\downsegment(\sigma)$, and in order to simplify notation we shall write $\cover\sigma$ instead of $\downsegment(\sigma)$. Note that we have
\[\cover{\sigma\tau} = \cover\sigma \cover\tau\;.\]

\begin{remark}
We remark that $\cover Q$ may be thought of as a generalized concept of
spectrum for weakly multiplicative open quantal frames (and for open groupoids), where the local bisections play the role of ``points''. In particular, if $\rs(Q)$ is the singleton locale $\{0,1\}$ (or, more generally, $\Omega$ in the underlying topos), it is easy to see that local bisections correspond to actual points and $j(a)$ is, for each $a\in Q$, the set of points ``in'' $a$.
\end{remark}

\begin{definition}
We say that $Q$ is \emph{weakly embeddable}, if
\[ \cover\sigma j(a) = j(\sigma\cdot a)\]
for all $a \in Q$ and $\sigma \in \sections(Q)$.
\end{definition}
 
Here $\sigma\cdot a$ denotes the action of $\sigma$ on $a$ as in \ref{definition:action}:

\[\sigma\cdot a = \V_{x^*y\le a}s^*(x)\wedge y\;.\]

\begin{remark}\label{mathfraki}
The completion $\cover {(-)}:\sections(Q)\to\cover Q$ defines an isomorphism $\sections(Q)\cong\ipi(\cover Q)$.
\end{remark}

\begin{lemma} \label{lemma:propertiesgamma}
If $Q$ is weakly embeddable the following properties hold for all $a,b \in Q$ and $\sigma,\tau \in \sections(Q)$:
\begin{enumerate}
\item\label{one} $j(a)j(b) \leq j(ab)$;
\item\label{two} $j(a^*) = j(a)^*$;
\item\label{three} $j(\sigma \cdot a) = \cover\sigma j(a)$ and $j(a\cdot \tau) = j(a)\cover\tau$;
\item\label{four} If $\cover\sigma  \leq j(a)$ then $\cover{\sigma\tau} \leq j(a\cdot \tau)$ and $\cover{\tau\sigma} \leq j(\tau \cdot a)$;
\item\label{five} If $j$ is mono and $\cover\sigma  \leq j(a)$ then $\sigma \cdot b \leq ab$;
\item\label{six} $j(\rs(Q)) = \rs(\cover Q) = j(E(\sections(Q))\cdot 1)$;
\item\label{seven} If $j$ is mono then $\sigma\cdot(\tau \cdot a) = (\sigma \cdot \tau)\cdot a$;
\item\label{eight} If $j$ is mono then $\sigma\cdot(a\cdot \tau) = (\sigma \cdot a)\cdot \tau$.
\end{enumerate}
\end{lemma}

Notice that \ref{four} above implies that $j(Q)$ is an involutive ideal of $\cover Q$.

\begin{proof}
\ref{one}. In order to see that we have $ j(a)j(b) \leq j(ab)$
let $\rho \in j(a)j(b)$. We can assume $\rho = \sigma\tau$  with $\sigma \in j(a)$ and
$\tau \in j(b)$, since a general element of $j(a)j(b)$ is the join of a compatible set of elements of this form. Then
\begin{eqnarray*}
U_{\sigma\tau} &=& s_{\sigma}^{*}(U_{\tau}^*)\\
&=& U_{\sigma} \wedge s_{\sigma}^{*}(U_{\tau}^*)\\
&=& s_{\sigma}^*(a) \wedge  s^{*}((s_{\tau}^{*}(b))^{*}) \leq s_{\sigma\tau}^{*}(ab)\;,
\end{eqnarray*}
and thus $\rho \in j(ab)$.

\ref{two}. For this property we show that $j(a^*) = (j(a))^{-1} = \{\sigma^{-1}\mid\sigma \in j(a)\}$.
Suppose that $\sigma \in j(a^*)$. Then $s^*(a^*) = U$.
We have $\sigma = (\sigma^{-1})^{-1}$. Hence we must show that $\sigma^{-1}\in j(a)$.
We have
\[
s_{\sigma^{-1}}^* (a) = t^*(a^*) = (\alpha^{-1})^*(s^*(a^*)) =  (\alpha^{-1})^*(U) = V\;,
\]
and thus  $\sigma^{-1}\in j(a)$.
On the other hand, suppose we have $\sigma \in j(a)$. Then $s^*(a) = U$.
Hence,
\[s_{\sigma^{-1}}^* (a^*) = t^*(a^{**}) =  (\alpha^{-1})^*( s^*(a)) =  \alpha^{-1*}(U) = V\;,\]
and thus $\sigma^{-1} \in j(a^*)$.

\ref{three}.  The first equality follows directly from weak embeddability. The second one follows from
\[j(a \cdot \tau) = j((\tau^{-1}\cdot a^*)^*) = (\tau^{*}j(a^*))^* = j(a^*)^*(\cover\tau ^{-1})^{-1} = j(a)\cover\tau \;,\]
using \ref{lemma:star} and \ref{two}.

\ref{four}. If $\cover\sigma  \leq j(a)$ then $\cover\sigma \cover\tau  \leq j(a)\cover\tau  = j(a \cdot \tau)$
using weak embeddability.

\ref{five}. If $\cover\sigma \leq j(a)$, then $j(\sigma \cdot b) = \cover\sigma j(b)  \leq j(a)j(b) \leq j(ab)$.
Since we are assuming that $j$ is mono, we get $\sigma \cdot b \leq ab$.

\ref{six}.  We have
\[j(z)1 = j(z)j(1_Q) \leq j(z1_Q) = j(z)\]
for $z \in \rs(Q)$. Hence $j(\rs(Q))$ consists of right-sided elements.
On the other hand, if $z \in \rs( \cover Q))$
 we get that $z = \cover\sigma 1 = \cover\sigma j(1_Q)$ for some $\sigma \in E(\sections(Q))$.
Then, using \ref{three}, we conclude that
\[z = \cover\sigma j(1_Q) = j(\sigma \cdot 1_Q) \in j(E(\sections(Q))\cdot 1_Q) \subset j(\rs(Q))\;.\]

\ref{seven}. We have $\cover\sigma (\cover\tau  j(a)) = j(\sigma\cdot(\tau\cdot a))$, by weak embeddability.
But also  $\cover\sigma (\cover\tau  j(a)) = (\cover\sigma  \cover\tau )j(a) = j((\sigma\tau)\cdot a)$.
Hence, since $j$ is assumed to be mono we get $\sigma\cdot(\tau\cdot a) = (\sigma\tau)\cdot a$.

\ref{eight}. From the fact that $(\cover\sigma j(a))\cover\tau  = \cover\sigma (j(a) \cover\tau )$,
and using analogous reasoning to the one used in \ref{seven}, we get the desired result. \qed
\end{proof}

The following lemma will be used further ahead.

\begin{lemma}\label{J}
Assume that
for all $\sigma,\tau \in \sections(Q)$ the condition $\cover\sigma  \leq j(a)$ implies
that $\cover{\tau\sigma} \leq j(\tau\cdot a)$, and that
$\sigma\cdot(\tau\cdot a) = (\sigma\tau)\cdot a$ holds. Then $Q$ is weakly embeddable.
\end{lemma}

\begin{proof}
Suppose that we have
$\cover\tau  \leq j(\sigma \cdot a)$.
Then $\tau \in j(\sigma \cdot a)$. But this means that
\begin{eqnarray*} U_{\tau} &\leq& s_{\tau}^*(\sigma \cdot a) = s_{\tau}^*\left(\V_{x^*y \leq a} s_{\sigma}^*(x) \wedge y\right)\\
&=& \V_{x^*y \leq a} s_{\sigma}^*(x) \wedge U_{\tau} \wedge s_{\sigma}^*(y) \leq \V_{x^*y \leq a}  s_{\sigma}^*(x) \leq U_{\sigma}\;.
\end{eqnarray*}
Hence,
$U_{\tau} \leq U_{\sigma}$ and,
bearing in mind that $\sigma\sigma^{-1} = (U_\sigma, u\circ k_{U_\sigma})$, 
we have
\[\tau = \sigma\sigma^{-1}\tau\;.\]
But we also have
\[\cover{\sigma^{-1} \tau} \leq  \cover{\sigma^{-1}}j(\sigma \cdot a) \leq j(\sigma^{-1}\sigma \cdot a) \leq j(\epsilon \cdot a) = j(a)\;,\]
using \ref{lemma:properties}.
Hence, $\cover\tau   = \cover{\sigma\sigma^{-1}\tau} \leq \cover\sigma  j(a)$, so that
\[j(\sigma \cdot a) \leq \cover\sigma  j(a)\;.\]
The other inequality follows from the hypothesis. \qed
\end{proof}

The quantale $\cover Q$ has the induced structure of an $\rs(Q)$-module whose action
is given by $z\cdot a = j(z) \wedge a$, for $z \in \rs(Q)$ and $a \in  \cover Q$.
We see that this induced module structure is precisely the natural $\rs( \cover Q)$-module structure of
$\cover Q$ taking into account the identification
\[\eta: \rs(Q) \rightarrow E(\sections(Q))1 = \spp( \cover Q)1 = \rs(\cover Q)\]
given by
$\eta (z) = \sigma\sigma^{-1}1$ for any $\sigma = (s,z)$. It is well defined and it does not depend on $\sigma$.

We have that
\begin{lemma}
Considering
the usual $\rs(Q)$-module structure and $\cover Q$ with
the one described above, $j$ is a homomorphism of $\rs(Q)$-modules.
\end{lemma}

\begin{proof}
\begin{eqnarray*}
j( z \wedge a) &=& j(\sigma\sigma^{-1}\cdot 1_Q \wedge a) = j(\sigma\sigma^{-1}\cdot 1_Q) \wedge j(a)\\
&=& \sigma\sigma^{-1}j(1_Q)\wedge j(a) = \sigma\sigma^{-1}1 \wedge j(a)\\
& =& \eta(z)\wedge j(a) = \eta(z)\cdot j(a)\;. \qed
\end{eqnarray*}
\end{proof}

\paragraph{(Strong) embeddability.} From now on we shall assume that $Q$ is not only weakly multiplicative but also weakly embeddable, and we shall address another condition on $j$ which, as we shall see, implies multiplicativity. First, an approximation of the envisaged condition is the following:

\begin{definition}
$Q$
is said to \emph{have enough bisections} if $j$ is mono.
A semiopen groupoid $G$ is said to \emph{have enough bisections} if $\opens(G)$ has enough
bisections. 
\end{definition}

\begin{example}
If $\rs(Q)= \{0,1\}$, having enough bisections means being spatial as a locale.
\end{example}

Now we introduce the slightly stronger
condition that we need.
Notice that $j$, being a homomorphism of $\rs(Q)$-modules, induces a homomorphism
$j \otimes j = (\ident \otimes j) \circ (j \otimes \ident)$ given as the composition:
\begin{equation}
\vcenter{\xymatrix{
Q\otimes_{\rs(Q)}Q\ar[rr]^{j \otimes \ident }&& \cover Q
\otimes_{\rs(Q)}Q \ar[rr]^{\ident \otimes j} &&\cover Q\otimes_{\rs(Q)}\cover Q \;.
 }}
\end{equation}

Since $\cover Q$ is an inverse quantal frame, it is, by \cite{RR}, a flat $\rs(Q)$-module.
Hence, we conclude that $\ident \otimes j$ is mono, and thus if $j \otimes \ident$ is mono $j \otimes j$ is also mono.

\begin{definition}
We say that $Q$ is \emph{embeddable} if it is weakly embeddable and $j\otimes \ident$ is mono.
A semiopen groupoid $G$ is said to be \emph{embeddable} if $\opens(G)$ is
embeddable.
\end{definition}

\begin{lemma}\label{lemma:enough}
If $Q$ is embeddable it has enough bisections.
\end{lemma}

\begin{proof}
Embeddability implies that $j\otimes j$ is mono. The result then follows 
from the fact that we have the commutative diagram
\begin{equation}
\vcenter{\xymatrix{
Q\otimes_{\rs(Q)}Q\ar[rr]^{j \otimes j }&& \cover Q 
\otimes_{\rs(Q)}\cover Q\\
Q\ar[u]_{\ident \otimes 1}\ar[rr]_{j}&& \cover Q\ar[u]_{\ident \otimes 1}
 }}
\end{equation}
and that $\ident \otimes 1$ on both sides is split mono because we have 
$[\ident,\ident]\circ \ident \otimes 1 = \ident_Q$. \qed
\end{proof}

\begin{lemma}
Let $Q$ have enough bisections. Then $\sigma\cdot(-)$ defines an involutive
$\sections(Q)$-action on $Q$.
\end{lemma}

\begin{proof}
This follows from \ref{lemma:propertiesgamma}-(\ref{seven},\ref{eight}). \qed
\end{proof}

\begin{remark}
Let $Q$ be embeddable. Then, by the above lemma  $\sigma\cdot(-)$ defines
an involutive $\sections(Q)$-action on $Q$. Consider $\cover Q$ with the
natural involutive $\sections(Q) \cong \ipi(Q)$-action given by $\sigma\cdot a = \cover\sigma  a$. Weak embeddability says that $j$ is a homomorphism of $\sections(Q)$-actions.
\end{remark}

We now come to our main result.

\begin{theorem}\label{maintheorem}
If $Q$ is embeddable it is multiplicative.
\end{theorem}

\begin{proof}
The proof of this result follows the same lines as that of \ref{ideal2}, whose notation we partly imitate. Let $\mu$ be the reduced multiplication of $\cover Q$ and $\imu$ the reduced multiplication of $Q$ (which plays a role analogous to that of the involutive ideal $\ideal$ in \ref{ideal2}). We begin by showing that
\[\mu_{*}(j(x)) = j \otimes j\circ\imu_*(x)\]
for any $x \in Q$. The inequality $\mu_{*}(j(x)) \geq j \otimes j \circ\imu_*(x)$ is immediate, since $j(a)j(b) \leq j(ab)$ for all $a,b\in Q$, so we need only show
that
\[\mu_{*}(j(x)) \leq  j\otimes j \circ\imu_*(x)\;.\]
Now, by hypothesis,
\[\V_{\scriptsize\begin{array}{c}y\in Q\\ yy^{*}y \leq x\end{array}}y \geq x\]
for all $x \in Q$, so that
\[\mu_{*}(j(x)) \leq \mu_{*}\left(\V_{\scriptsize\begin{array}{c}y\in Q\\ yy^{*}y \leq x\end{array}}j(y)\right) = \V_{\scriptsize\begin{array}{c}y\in Q\\ yy^{*}y \leq x\end{array}}\mu_{*}(j(y))\;,\]
using the fact that $\mu_*$ preserves joins because $\cover Q$ is inverse.
Hence we need only show that
\[\mu_{*}(j(y)) \leq  \imu_*(j(x))\]
for all $y\in Q$ such that $yy^{*}y \leq x$ and the result will follow by taking the supremum.
Let then $y$ be such an element.
We have
\[\mu_*(j(y)) = \V_{ab \leq j(y)} a \otimes b\;.\]
But since $\V \ipi(\cover Q) = 1$ we also get that $\V \ipi(\cover Q) \otimes \ipi(\cover Q) = 1 \otimes 1$,
so that if we show that
\[s \otimes t \leq \imu_{*}(j(x))\]
for all $s,t \in \ipi(\cover Q)$ such that $st \leq j(y)$ the result will follow by taking the supremum of all such pure tensors.
Let $s,t\in\ipi(\cover Q)$ be such elements. Then we have that
\[s \otimes t = ss^{*}s \otimes t = s \otimes s^{*}st \leq s\otimes s^{*}j(y)\]
and, analogously,
\[s \otimes t \leq j(y)t^{*} \otimes t\;.\]
Hence,
\[s \otimes t \leq (s \wedge yt^{*}) \otimes (t \wedge s^{*}j(y)) \leq j(y)t^{*} \otimes s^{*}j(y)\;.\]
Now, taking into account the isomorphism $\sections(Q)\cong\ipi(\cover Q)$ (\cf\ \ref{mathfraki}), there must be $\sigma,\tau\in\sections(Q)$ such that $s=\cover\sigma$ and $t=\cover\tau$,
and thus by \ref{lemma:propertiesgamma}-\ref{three} we have
\begin{equation}\label{eqjyt}
j(y)t^*=j(y)\cover{\tau^{-1}} = j(y\cdot \tau^{-1}) \in j(Q)
\end{equation}
and
\begin{equation}\label{eqsjy}
s^*j(y)=\cover{\sigma^{-1}}j(y) = j(\sigma^{-1}\cdot y) \in j(Q)\;.
\end{equation}
In addition, by \ref{lemma:propertiesgamma}-\ref{two} we get
$\cover{\tau^{-1}\sigma^{-1}}=(st)^* \leq j(y^*)$.
Since by \ref{lemma:enough} $j$ is
mono, \ref{lemma:propertiesgamma}-\ref{five} gives us $y \cdot (\tau^{-1} \sigma^{-1}) \leq yy^*$
and we get
\[(y\cdot\tau^{-1})(\sigma^{-1}\cdot y) = (y\cdot(\tau^{-1} \sigma^{-1}))y \leq yy^{*}y \leq x\;,\]
using the properties of involutive actions of \ref{propinvo}.
Hence, using (\ref{eqjyt}) and (\ref{eqsjy}) we obtain
\[(j(y)t^*)(s^*j(y))=j(y\cdot\tau^{-1})j(\sigma^{-1}\cdot y)\le j(x)\;,\]
and also
\[s \otimes t \leq j(y)t^*\otimes s^*j(y) \le j\otimes j \circ \imu_{*}(x)\;.\]
Taking the supremum of the pure tensors $s\otimes t$ we conclude $\mu_{*}(j(x)) \leq  j\otimes j \circ\imu_*(x)$ as desired.

Now we show that $\imu_*$ preserves joins. 
Consider a join $\V_\alpha x_\alpha$ with $x_\alpha \in Q$. Applying what we proved we have, since $\V_\alpha x_\alpha \in Q$:
\begin{eqnarray*}
j\otimes j\left(\imu_*\left(\V_\alpha x_\alpha\right)\right) &=& \mu_*\left(j\left(\V_\alpha x_\alpha\right)\right) = \V_\alpha \mu_*(j(x_\alpha))\\
&=& \V_\alpha j\otimes j(\imu_{*}(x_\alpha))
=j\otimes j\left(\V_\alpha \imu_{*}(x_\alpha)\right)\;.
\end{eqnarray*}
Since $j \otimes j$ is mono, $Q$ is multiplicative. \qed
\end{proof}

\paragraph{Coverable open groupoids.}

As we have seen, weak multiplicativity plus embeddability is a sufficient, but not necessary, condition for multiplicativity. Hence, the quantale $\opens(G)$ of an open groupoid $G$ is not necessarily embeddable. In those examples where it is, we may regard the embedding $j:\opens(G)\to\cover{\opens(G)}$ dually as some kind of cover of $G$. We conclude this paper by briefly studying such groupoids. We shall see
that these include many of the examples that occur
in practice, in particular Lie groupoids.

\begin{definition}
Let $G$ be an open (localic) groupoid. We denote by $\cover G$ the localic \'etale groupoid $\groupoid(\cover{\opens(G)})$.
The groupoid $G$ is said to be \emph{coverable} (resp.\ \emph{weakly coverable}) if its quantale $\opens(G)$ is embeddable (resp.\ \emph{weakly embeddable}). We also say that $G$ \emph{has enough (local) bisections} if $\opens(G)$ does.
\end{definition}

Our terminology is justified because (weakly) coverable groupoids are covered by \'etale groupoids in the following sense:

\begin{theorem}
Let $G$ be a weakly coverable open (localic) groupoid with enough bisections (for instance, a coverable groupoid). A functor of localic groupoids
\[J:\cover G\to G\]
is defined where $J_0$ is the canonical isomorphism $\cover G_0\to G_0$ and $J_1:\cover G_1\to G_1$ is given by $J_1^* = j$.
Moreover, $J$ is an epimorphism, and it is an isomorphism if and only if $G$ is \'etale.
\end{theorem}

\begin{proof}
Let $\cover m$ and $m$ be the multiplications of $\cover G$ and $G$, respectively. In order to show that $J$ is a functor we essentially need to show that
\[ {\cover m}^{*}\circ j = j \otimes j \circ m^{*}\;,\]
which follows analogously to the first part of the proof of Theorem 
\ref{maintheorem} (with ${\cover m}^*$ and $m^*$ playing the role of $\mu_*$ and $\imu_*$, respectively).
If $G$ is \'etale, $\opens(G)$ is an inverse quantal frame. Hence, in this case $J$ is an isomorphism because, identifying $\sections(\opens(G))$ with $\ipi(\opens(G))$, the homomorphism $j$ is the canonical isomorphism
\[\opens(G)\stackrel{\cong}\to\lcc(\ipi(\opens(G)))\;.\]
Finally, $J_1$ is an epimorphism of locales, and $J$ is an epimorphism of groupoids, because $G$ has enough sections, which means that $j$ is a monomorphism. \qed
\end{proof}

We remark that $j$ is not necessarily a homomorphism of quantales. In other words, the covering of $G$ is better behaved than the embedding of $\opens(G)$, and provides an example of a situation where the dual of a homomorphism of groupoids is not a homomorphism of quantales (\cf\ \cite{aim}).  (However, if $j$ is a homomorphism --- for which it suffices to require the condition $j(ab)\le j(a)j(b)$ --- \cf\ \ref{lemma:propertiesgamma}-\ref{one} --- then $\opens(G)$ is isomorphic as a quantale to $j(\opens(G))$, which is an involutive ideal of $\cover{\opens(G)}$.)

In order to find examples of coverable groupoids we shall look at sober topological groupoids. For such a groupoid $G$ we shall write $\cover G$ for the spectrum of $\groupoid(\cover{\opens(G)})$; that is, $\cover G$ is the groupoid of germs of local bisections of $G$ or, in other words, the \'etale groupoid associated to the abstract complete pseudogroup $\sections(G)$ as in \cite{MatsnevResende}. Some remarks are immediate:
\begin{enumerate}
\item $\cover G$ is sober;
\item $G$ has enough bisections if and only if for every $x \in G_1$
there is a local bisection $(s,U)$ such that $x\in s(U)$;
\end{enumerate}

We remark that a sober topological open groupoid $G$ is not ``the same'' as a localic open groupoid because the quotient of frames \[\topology(G_1)\otimes_{\topology(G_0)}\topology(G_1)\to\topology(G_1\times_{G_0} G_1)\] is not necessarily an isomorphism. However, a sufficient condition for this quotient to be an isomorphism is to have $G_1$ locally compact (\cf\ \cite[p.\ 61]{Johnstone}), and this includes many examples in practice. For instance,
the usual notion of locally compact groupoid stems from harmonic analysis and carries more information than mere local compactness of $G_1$. (A locally compact groupoid $G$ is usually also open, either by definition or as a consequence of measure-theoretic constraints, it is often second-countable, $G_0$ is usually assumed to be Hausdorff, and in those cases where $G_1$ is not Hausdorff it is required to satisfy a condition which in fact is stronger than local compactness.)

For the purposes of this paper it suffices to adopt the following very general definition:

\begin{definition}
A topological groupoid $G$ will be said to be \emph{locally compact} if it is open and $G_1$ is a locally compact space.
\end{definition}

\begin{example}
Lie groupoids are sober because they are Hausdorff, they are locally compact because they are manifolds, and they have enough bisections due to the local triviality of $d$. See \cite{MoerdijkMrcun,Paterson}.
\end{example}

\begin{theorem}
Any sober locally compact groupoid with enough bisections is coverable.
\end{theorem}

\begin{proof}
Let $G$ be a locally compact groupoid with enough bisections which is also sober (\ie, $G_1$ is a sober space). We prove that $\opens(G)$ is weakly embeddable by verifying the two hypotheses of 
\ref{J}.
We know that $\sigma\cdot(-)$ induces an involutive action
on $\opens(G)$ (\cf\ \ref{exm:GQactiontopol}) so we need only show that the first hypothesis of \ref{J} holds.
Let $\tau \in j(W)$ for some open set $W$ of $G_1$. This means that $s_{\tau}(U) \subset W$.
Now
\[s_{\sigma\tau}(x) =  s_{\sigma}(x) s_{\tau}(r(s_{\sigma}(x))) \subset s_{\sigma}(x) W\;,\]
for all $x \in s_{\sigma}^{-1}r^{-1}(U)$.
But, for the same $x$,  $s_{\sigma}(x)W$ consists of elements $s_{\sigma}(x)  y$
such that $y \in W$ and $r(s_{\sigma}(x)) = d(y)$. That is to say
\[y \in d^{-1}(\{r(s_{\sigma}(x))\}) \subset d^{-1}(V)\]
and, since $r(s_{\sigma}(x)) = \alpha_{\sigma}(x)$, we
obtain $x = \alpha_{\sigma}^{-1}(d(y))$, so that  $s_{\sigma}(x)W$ consists of elements
$s_{\sigma}( \alpha_{\sigma}^{-1}(d(y)))  y = t_{\sigma}(d(y)) y$ with
\[y \in d^{-1}(\{r(s_{\sigma}(x))\})\subset d^{-1}(V)\;.\]
Since $t_{\sigma}(d(y)) y\in\sigma\cdot W$, we conclude that
$\cover{\sigma\tau} \leq j(\sigma\cdot W)$,
and weak embeddability follows from \ref{J}.

Now we show that $j \otimes j$ is mono.
We have the continuous map of topological spaces
\[k: \cover G_1 \to G_1\]
given by $k(s_y) = s(y)$, where $s_y$ denotes the germ at $y \in U$ of the local bisection $\sigma = (s,U)$. Having enough bisections implies that $k$ is surjective.
It is clear that the inverse image frame homomorphism $k^{-1}$ is $j$,
and thus we have the following commutative diagram of frame homomorphisms where the left vertical arrow is an isomorphism because $G_1$ is locally compact:
\[
\xymatrix{
\topology(G_1)\otimes_{\topology(G_0)}\topology(G_1)\ar@{->>}[d]_{\cong}\ar[rr]^-{j\otimes j}&&
\topology(\cover G_1)\otimes_{\topology(G_0)}\topology(\cover G_1)\ar@{->>}[d]\\
\topology(G_2)\ar[rr]^-{(k\times k)^{-1}}&&\topology(\cover G_2)\;.
}
\]
Hence, $j\otimes j$ is mono because $(k\times k)^{-1}$ is. \qed
\end{proof}

\begin{corollary}
Every Lie groupoid is coverable.
\end{corollary}

~\\
\noindent {\sc
Faculdade de Engenharia e Ci\^{e}ncias Naturais\\
Universidade Lus\'ofona\\
Campo Grande 376, 1749-024 Lisboa, Portugal\\
{\it E-mail:} {\sf cprotin@sapo.pt}\\
~\\
Departamento de Matem\'{a}tica, Instituto Superior T\'{e}cnico\\
Universidade T\'{e}cnica de Lisboa\\
Av.\ Rovisco Pais 1, 1049-001 Lisboa, Portugal\\
{\it E-mail:} {\sf pmr@math.ist.utl.pt}\\
}

\end{document}